\documentclass[12pt]{article}
\usepackage{amsmath,amssymb,amsfonts,amsthm,latexsym,graphicx,hyperref,footnote,enumerate}
\usepackage{booktabs}

\usepackage{setspace}
\usepackage{tabularx}
\usepackage{multirow, makecell}
\usepackage{rotating}
\usepackage{floatrow, multicol}

\usepackage[noadjust]{cite}

\newtheorem{thm}{Theorem}[section]
\newtheorem{lem}[thm]{Lemma}
\newtheorem{cor}[thm]{Corollary}

\theoremstyle{definition}
\newtheorem{remark}[thm]{Remark}
\newtheorem{exa}[thm]{Example}

\newcommand{\tabincell}[2]{\begin{tabular}{@{}#1@{}}#2\end{tabular}}
\renewcommand\proofname{\textit{Proof}}
\addtocounter{section}{0}
\allowdisplaybreaks[4]

\title{\bf \Large A complete characterization of graphs with exactly two positive eigenvalues}
\author{
{\small  Fang Duan$^a$, \ \ Qiongxiang Huang$^b$, \ \ Xueyi Huang$^c$, \ \ Zoran Stani\'c$^d$, \ \  Jianfeng Wang$^{e,}$\footnote{Corresponding author.
\newline{\it \hspace*{5mm}Email addresses:} fangbing327@126.com (F. Duan), huangqx@xju.edu.cn (Q.X. Huang), huangxymath@163.com (X.Y. Huang),
zstanic@matf.bg.ac.rs (Z. Stani\'c), jfwang@sdut.edu.cn (J.F.Wang).}}\\[2mm]
\footnotesize $^a$School of Mathematics Science, Xinjiang Normal University, Urumqi 830017, China\\
\footnotesize $^b$College of Mathematics and Systems Science, Xinjiang University, Urumqi  830046,  China\\
\footnotesize $^c$School of Mathematics, East China University of Science and Technology, Shanghai 200237, China\\
\footnotesize $^d$Faculty of Mathematics, University of Belgrade, Studentski trg 16, 11 000 Belgrade, Serbia\\
\footnotesize $^e$School of Mathematics and Statistics, Shandong University of Technology, Zibo 255049, China}
\date{ }
\date{}

\begin{document}
\maketitle

\begin{abstract}
In 1977 Smith characterized graphs with exactly one positive eigenvalue. Since then, many particular results related to graphs with exactly two positive eigenvalues have emerged. In this paper we conclude this investigation by giving a full characterization of these graphs.

\medskip


\end{abstract}

\baselineskip=0.21in

\section{Introduction}
For a simple undirected graph $G$ we write $n, V(G), E(G)$ and $A(G)$ to denote the number of its vertices (also known as the \textit{order}), its vertex set, edge set and the standard adjacency matrix, respectively. The eigenvalues of $A(G)$, denoted by $\lambda_1(G)\geq  \lambda_2(G)\geq  \cdots \geq \lambda_n(G)$, are known as the \emph{eigenvalues} of $G$. Together with their repetitions they form the \emph{spectrum} of $G$ denoted by $\mathrm{Spec}(G)$. The number of positive eigenvalues $p(G)$, the number of negative eigenvalues $n(G)$ and the number of zero eigenvalues $\eta(G)$ of $G$ are called the \emph{positive inertia index}, the \emph{negative inertia index} and the \emph{nullity} of $G$, respectively. Evidently, we have $p(G)+n(G)+\eta(G)=n$.

The invariants $p(G)$, $n(G)$ and $\eta(G)$ have been broadly investigated in the last decades. For some results on this topic, we refer the reader to \cite{F.Duan, H.C.Ma, M.R.Oboudi1, M.R.Oboudi3,M.Petrovi'c, A.Torga, X.L.Wang, G.H.Yu1}, and references therein. In particular, the problem of  characterizing  graphs with a comparatively small positive inertia index
originates from the work of Smith \cite{J.H.Smith} who proved that a graph has exactly one positive eigenvalue if and only if it is the disjoint union of a complete multipartite graph and some isolated vertices. Graphs with at most two non-negative eigenvalues are determined by Petrovi\' c~\cite{M.Petrovi'c} (see also Obudi's~\cite{M.R.Oboudi3}). The next natural step is a characterization of graphs with exactly two positive eigenvalues. The main challenge in this task lies in fact that such graphs are partitioned into (a) a comparatively large number of infinite families determined by various structural parameters and (b) additional individual graphs. Therefore, so far we had only  sporadic results on this topic, and some of them can be found in \cite{Bei,the}; for example, all bipartite graphs and all line graphs with exactly two positive eigenvalues are known.  To list more known results, we need to introduce some notation. Let $\mathcal{T}_n$ denote the collection of $n$-vertex graphs with exactly two positive eigenvalues. For $G\in \mathcal{T}_n$, we have $n(G)\geq 1$, and so $\eta(G)\leq n-3$. Accordingly, $\mathcal{T}_n$ can be partitioned as $\mathcal{T}_n=\mathcal{T}^0_n\cup \mathcal{T}^1_n \cup\cdots \cup \mathcal{T}^{n-3}_n$, where $\mathcal{T}^i_n=\{G\in \mathcal{T}_n\mid \eta(G)=i\}$, for $0\leq i\leq n-3$. In the past four decades, efforts of several researchers have determined some subclasses of $\mathcal{T}_n$. Petrovi\'{c}~\cite{M.Petrovi'c} characterized  graphs of $\mathcal{T}^0_n$ by means of induced subgraphs. A~different characterization of the same graphs is given by Oboudi~\cite{M.R.Oboudi3}. Following the latter reference, the authors of ~\cite{F.Duan} reported a characterization of graphs belonging to~$\mathcal{T}^1_n$.

In this paper we conclude this investigation by giving a full recursive characterization of all infinite families of graphs belonging to the class $\mathcal{T}^i_n=\{G\in \mathcal{T}_n\mid \eta(G)=i\},~ 2\leq i\leq n-3$, and by determining all individual graphs of the same class. In the latter case we use a mixture of a theoretical consideration and computer search. Namely, we use a theoretical approach to narrow down the search on such graphs by limiting their order to be less than or equal to 14 and by proving that they belong to certain structural types. Then we use the computer to obtain all of them.

To state our contribution we need some notation. We denote by $G+H$ the disjoint union of  graphs $G$ and $H$. We write  $K_{n_1, n_2, \ldots, n_l}$ for the complete multipartite graph with $l\geq 2$ parts of sizes $n_1, n_2, \ldots, n_l$, and $K_n$ for the complete graph with $n$ vertices. We give a recursive characterization of graphs of $\mathcal{T}^s_n$ for $2\leq s\leq n-3$ as follows. (We point out that definitions of congruent vertices of I-type, II-type and III-type are given in the next section.)
\begin{thm}\label{thm-1-2}
Let $G$ be a  graph of order $n\ge 6$.
\begin{enumerate}[(a)]
\item If $G$ is disconnected,  then $G\in \mathcal{T}_n^s~(2\leq s\leq n-3)$ if and only if $G\cong H+lK_1$, where $H\in \mathcal{T}^{s-l}_{n-l}$ is connected and $1\leq l\leq s$, or $G\cong K_{t_1, t_2, \ldots, t_p}+K_{r_1, r_2, \ldots, r_q}+mK_1$, where $m\geq 0$ and $(t_1-1)+(t_2-1)+\cdots +(t_p-1)+(r_1-1)+(r_2-1)+\cdots +(r_q-1)=s-m$.
\item If $G$ is connected, then
\begin{enumerate}[(b1)]
\item $G\in \mathcal{T}_n^2$ if and only if $G$ is one of the $175$ graphs (called underivable graphs) listed in Table~\ref{tab-3}, or $G$ is obtained by adding a congruent vertex of I-type, II-type or III-type to a graph $G'\in \mathcal{T}_{n-1}^1$, and
\item  $G\in \mathcal{T}_n^s$ for $3\leq s \leq n-3$ if and only if $G$ obtained by adding a congruent vertex of I-type, II-type or III-type to a graph $G'\in\mathcal{T}_{n-1}^{s-1}$.
\end{enumerate}
\end{enumerate}
\end{thm}

Together with results of \cite{F.Duan} and \cite{M.R.Oboudi3}, Theorem~\ref{thm-1-2} gives a full characterization of graphs with exactly two positive eigenvalues.

%\begin{thm}\label{thm-1-1}
%Let $G$ be a disconnected graph of order $n\ge 6$. Then, for $2\leq s\leq n-3$,  $G\in \mathcal{T}_n^s$  if and only if $G\cong H+lK_1$, where $H\in \mathcal{T}^{s-l}_{n-l}$ is connected and $1\leq l\leq s$, or $G\cong K_{t_1, t_2, \ldots, t_p}+K_{r_1, r_2, \ldots, r_q}+mK_1$, where $m\geq 0$ and $[(t_1-1)+(t_2-1)+\cdots +(t_p-1)]+[(r_1-1)+(r_2-1)+\cdots +(r_q-1)]=s-m$.
%\end{thm}

%For a connected graph of $\mathcal{T}_n^s$ $(2\leq s\leq n-3)$, we introduce three types of congruent vertices defined in section 2 and we show that except for 175 specific connected graphs selected in $\mathcal{T}_n^2$, which is called nonconstructible graphs, the other connected graphs of $\mathcal{T}_n^s$ can be recursively constructed from the graphs in $\mathcal{T}^{s-1}_{n-1}$ by three types of congruent vertices in the following theorem.

%\begin{thm}\label{thm-1-2}
%Let $G$ be a connected graph of order $n\ge 6$. Then
%\begin{enumerate}[(i)]
%\item  $G\in \mathcal{T}_n^2$ if and only if $G$ is one of the 175 graphs shown in Table 3, or $G$ comes from $G'\in \mathcal{T}_{n-1}^1$ by adding a congruent vertex of I-, II- or III-type (see Section 2 for the definition);
%\item  for $3\leq s \leq n-3$, $G\in \mathcal{T}_n^s$ if and only if $G$ comes from $G'\in\mathcal{T}_{n-1}^{s-1}$ by adding a congruent vertex of I-, II- or III-type.
%\end{enumerate}
%\end{thm}

The remainder of the paper is organized as follows. In Section~\ref{sec:prel} we list some known results and fix additional terminology and notation. For the sake of completeness, the main results of \cite{M.R.Oboudi3, F.Duan} concerning graphs of $\mathcal{T}^0_n\cup \mathcal{T}^1_n$ are reviewed in Section~\ref{sec:3}.  Section~\ref{sec:4} contains the proof of Theorem~\ref{thm-1-2}; the main contribution of this paper. A recapitulation of the entire paper is given in Section~\ref{sec:rec}.

\section{Preliminaries}\label{sec:prel}

We believe that the reader is familiar with fundamental concepts in spectral graph theory which, for example, include the eigenvalue interlacing. In what follows we list some basic results, mostly to make the paper more self-contained since all of them are also broadly known.

%
%
%\begin{thm}(Interlacing theorem, \cite{D.Cvetkovi})\label{thm-2-1}
%Let $G$ be a graph of order $n$ and $H$ be an induced subgraph of $G$ with order $m$. Suppose that $\lambda_1(G)\geq
%\cdots \geq \lambda_n(G)$ and $\lambda_1(H)\geq \cdots \geq \lambda_m(H)$ are the eigenvalues of $G$ and $H$, respectively. Then for every $1\leq
%i\leq m$, $\lambda_i(G)\geq \lambda_i(H)\geq \lambda_{n-m+i}(G)$.
%\end{thm}

\begin{lem}\cite{D.Cvetkovi}\label{lem-2-1}
Let $H$ be an induced subgraph of a graph $G$. Then $p(H)\leq p(G)$.
\end{lem}

A \textit{pendant vertex} of a graph is a vertex of degree 1.

\begin{lem}\cite{H.C.Ma} \label{lem-2-2}
Let $G$ be a graph containing a pendant vertex, and let $H$ be the induced subgraph of~$G$ obtained by deleting the pendant vertex together
with its unique neighbour. Then $p(G)=p(H)+1$, $n(G)=n(H)+1$ and $\eta(G)=\eta(H)$.
\end{lem}

%Let $B$ and $D$ be two real symmetric matrices of order $n$. Then $D$ is \emph{congruent} to $B$ if there exists a real invertible matrix $C$ such that $D=C^\intercal BC$. Traditionally, we say that $D$ is obtained from $B$ by a congruent transformation. The famous Sylvester's law of inertia states that the inertia index of two matrices is unchanged by congruent transformation.
%
%\begin{lem}\label{lem-2-3}
%(Sylvester's law of inertia) If two real symmetric matrices $A$ and $B$ are congruent, then they have the same positive (resp.~negative) inertia index and nullity.
%\end{lem}

We restate the aforementioned result of Smith.

\begin{thm}\cite{J.H.Smith} \label{thm-2-2}
A graph has exactly one positive eigenvalue if and only if its non-isolated vertices form a complete multipartite graph.
\end{thm}

We now proceed with some additional definitions, terminology and notation. In particular, we introduce the four types of graph transformations that will be frequently used in the forthcoming sections.

We write $C_n$ and $P_n$ for the cycle and the path with $n$ vertices and denote by $G-v$ (resp.~$G-e$) the graph obtained by deleting a vertex $v$ (resp.~edge~$e$) from a graph $G$.

For any $v\in V(G)$ and $W\subseteq V(G)$, we set $N_W(v)=\{u\in W\mid uv \in E(G)\}$. For $W=V(G)$ we write $N_G(v)$ or just $N(v)$ (if the graph is clear from the context) instead of $N_{V(G)}(v)$, and we write $N_G[v]$ or $N[v]$ for $N_G(v)\cup \{v\}$. Of course, $N(v)$ and $N[v]$ are known as the \emph{neighbourhood} and the \emph{closed neighbourhood} of $v$, respectively. The cardinality of the former set is known as the \textit{degree} of~$v$, denoted by $d(v)$. We also use $G[W]$ to denote the subgraph induced by the vertices of $W$.

Let $G$ be a graph with vertex set $\{v_1, v_2, \ldots, v_n\}$. We denote by $G[K_{t_1}, K_{t_2}, \ldots, K_{t_n}]$ the \emph{generalized lexicographic product} of
$G$ with $K_{t_1}, K_{t_2},\ldots, K_{t_n}$, that is a graph obtained from $G$ by replacing the vertex $v_j$ with $K_{t_j}$ and then joining every vertex of $K_{t_i}$ to every vertex of $K_{t_j}$ if and only if $v_i$ is adjacent to $v_j$ in $G$.

For the graphs $G_1$ and $G_2$, a vertex $u\in V(G_1)$ and an integer $k$ $(1\leq k \leq |V(G_2)|)$, the \emph{$k$-joining graph} $G_1(u)\odot^k G_2$ is  obtained from $G_1+G_2$ by joining $u$ to any $k$ vertices of $G_2$.

Following \cite{M.R.Oboudi3}  we introduce  a particular class of graphs. For an integer $n\geq 2$, let $K_{\lceil\frac{n}{2}\rceil}$ and $K_{\lfloor\frac{n}{2}\rfloor}$ be the vertex disjoint complete graphs with vertex sets $V=\{v_1, v_2, \ldots, v_{\lceil\frac{n}{2}\rceil}\}$ and $W=\{w_1, w_2,  \ldots, w_{\lfloor\frac{n}{2}\rfloor} \}$, respectively. The graph $G_n$ is defined as the graph obtained from $K_{\lceil\frac{n}{2}\rceil}+K_{\lfloor\frac{n}{2} \rfloor}$ by applying the following operations:
\begin{itemize}
\item For $n$ even, we insert new edges into $K_{\frac{n}{2}}+K_{\frac{n}{2}}$ in such a way that
\begin{equation*}\begin{aligned}
N_W(v_1)&=\emptyset\subset N_W(v_2)=\{w_{\frac{n}{2}}\}\subset N_W(v_3)=\{w_{\frac{n}{2}},  w_{\frac{n}{2}-1}\}\subset\cdots\subset N_W(v_{\frac{n}{2}-1})\\
&=\{w_{\frac{n}{2}}, w_{\frac{n}{2}-1}, \ldots, w_{3} \}\subset N_W(v_{\frac{n}{2}})=\{w_{\frac{n}{2}}, w_{\frac{n}{2}-1}, \ldots, w_{2} \}.
\end{aligned}\end{equation*}

\item For $n$ odd, we insert new edges into $K_{\frac{n+1}{2}}+K_{\frac{n-1}{2}}$ in such a way that
\begin{equation*}\begin{aligned}
N_W(v_1)&=\emptyset\subset N_W(v_2)=\{w_{\frac{n-1}{2}}\}\subset N_W(v_3)=\{w_{\frac{n-1}{2}}, w_{\frac{n-1}{2}-1}\}\subset\cdots\subset N_W(v_{\frac{n+1}{2}-1})\\
&=\{w_{\frac{n-1}{2}}, w_{\frac{n-1}{2}-1}, \ldots, w_{2} \}\subset N_W(v_{\frac{n+1}{2}})=\{w_{\frac{n-1}{2}}, w_{\frac{n-1}{2}-1}, \ldots, w_{1} \}.
\end{aligned}\end{equation*}
\end{itemize}

We call $G_n$ the \emph{reduced half-complete graph}. By definition, $G_2\cong 2K_1$, $G_3\cong P_3$ and $G_4\cong P_4$. The graphs $G_5$,  $G_6$ and, in general, $G_{2r}$, $G_{2r+1}$ are illustrated in Figure~\ref{fig-1}. Clearly, $G_n$ is obtained from $G_{n+1}$ by deleting the vertex with maximum (resp.~minimum) degree when $n$ is even (resp.~odd). For example, by observing Figure~~\ref{fig-1} we get that $G_{2r}$ is obtained by deleting the vertex $v_{r+1}$ of $G_{2r+1}$.

\begin{figure}
\begin{center}
\unitlength 1.3mm % = 2.845pt
\linethickness{0.4pt}
\ifx\plotpoint\undefined\newsavebox{\plotpoint}\fi % GNUPLOT compatibility
\begin{picture}(84.5,22.097)(0,0)
\put(28.225,15.142){\circle*{.94}}
\put(28.314,7.364){\circle*{.94}}
\put(28.374,15.026){\line(0,-1){7.69}}
\put(25.22,18.678){\circle*{.94}}
%\emline(25.104,18.826)(28.197,15.291)
\multiput(25.104,18.826)(.033619565,-.038423913){92}{\line(0,-1){.038423913}}
%\end
\put(21.772,15.054){\circle*{.94}}
\put(21.684,7.452){\circle*{.94}}
%\emline(25.103,18.826)(21.656,15.291)
\multiput(25.103,18.826)(-.033466019,-.034320388){103}{\line(0,-1){.034320388}}
%\end
\put(21.568,15.114){\line(0,-1){7.601}}
\put(21.568,7.513){\line(1,0){6.718}}
%\emline(21.656,15.114)(28.285,7.424)
\multiput(21.656,15.114)(.033649746,-.039035533){197}{\line(0,-1){.039035533}}
%\end
\put(21.656,15.114){\line(1,0){6.629}}
\put(25.131,3.829){\circle*{.94}}
%\emline(21.567,7.425)(25.103,3.889)
\multiput(21.567,7.425)(.03367619,-.03367619){105}{\line(0,-1){.03367619}}
%\end
%\emline(28.462,7.425)(25.191,3.889)
\multiput(28.462,7.425)(-.033721649,-.036453608){97}{\line(0,-1){.036453608}}
%\end
\put(25.103,5.304){\makebox(0,0)[cc]{\scriptsize$w_1$}}
\put(20.507,6.011){\makebox(0,0)[cc]{\scriptsize$w_2$}}
\put(28.727,5.746){\makebox(0,0)[cc]{\scriptsize$w_3$}}
\put(25.279,.177){\makebox(0,0)[cc]{\footnotesize$G_6$}}
\put(24.838,16.882){\makebox(0,0)[cc]{\scriptsize$v_1$}}
\put(24.97,16.838){\oval(12.286,5.392)[]}
\put(24.97,5.79){\oval(12.286,5.392)[]}
\put(9.31,15.143){\circle*{.94}}
\put(9.398,7.365){\circle*{.94}}
\put(9.457,15.027){\line(0,-1){7.69}}
\put(6.303,18.679){\circle*{.94}}
%\emline(6.187,18.827)(9.28,15.292)
\multiput(6.187,18.827)(.033619565,-.038423913){92}{\line(0,-1){.038423913}}
%\end
\put(2.856,15.055){\circle*{.94}}
\put(2.767,7.453){\circle*{.94}}
%\emline(6.186,18.827)(2.739,15.292)
\multiput(6.186,18.827)(-.033466019,-.034320388){103}{\line(0,-1){.034320388}}
%\end
\put(2.652,15.115){\line(0,-1){7.601}}
\put(2.652,7.514){\line(1,0){6.718}}
%\emline(2.739,15.115)(9.369,7.426)
\multiput(2.739,15.115)(.033654822,-.039030457){197}{\line(0,-1){.039030457}}
%\end
\put(2.739,15.115){\line(1,0){6.629}}
\put(6.143,16.75){\oval(12.286,5.392)[]}
\put(6.231,7.425){\oval(12.109,3.182)[]}
\put(6.276,.531){\makebox(0,0)[cc]{\footnotesize$G_5$}}
\put(6.099,17.148){\makebox(0,0)[cc]{\scriptsize$v_1$}}
\put(42.013,18.325){\circle*{.94}}
\put(52.089,18.502){\circle*{.94}}
\put(44.311,18.413){\circle*{.94}}
\put(42.013,4.89){\circle*{.94}}
\put(47.943,18.297){\circle*{.25}}
\put(45.645,4.862){\circle*{.25}}
\put(48.51,4.81){\circle*{.25}}
\put(50.233,4.89){\circle*{.94}}
%\emline(52.061,18.651)(44.371,4.95)
\multiput(52.061,18.651)(-.03372807,-.060092105){228}{\line(0,-1){.060092105}}
%\end
\put(46.344,18.413){\circle*{.94}}
\put(49.181,18.245){\circle*{.25}}
\put(52.089,4.89){\circle*{.94}}
\put(50.33,18.245){\circle*{.25}}
\put(47.06,4.81){\circle*{.25}}
\put(52.061,18.562){\line(0,-1){13.612}}
%\emline(52.061,5.039)(46.404,18.297)
\multiput(52.061,5.039)(-.033672619,.078916667){168}{\line(0,1){.078916667}}
%\end
%\emline(50.205,4.95)(46.227,18.385)
\multiput(50.205,4.95)(-.033711864,.113855932){118}{\line(0,1){.113855932}}
%\end
%\emline(44.194,18.474)(52.061,4.862)
\multiput(44.194,18.474)(.033619658,-.05817094){234}{\line(0,-1){.05817094}}
%\end
%\emline(51.973,18.562)(50.117,5.127)
\multiput(51.973,18.562)(-.03314286,-.23991071){56}{\line(0,-1){.23991071}}
%\end
\put(41.72,3.713){\makebox(0,0)[cc]{\scriptsize$w_1$}}
\put(44.371,3.713){\makebox(0,0)[cc]{\scriptsize$w_2$}}
\put(44.283,19.446){\makebox(0,0)[cc]{\scriptsize$v_2$}}
\put(41.808,19.358){\makebox(0,0)[cc]{\scriptsize$v_1$}}
\put(44.577,4.89){\circle*{.94}}
\put(37.831,17.325){\makebox(0,0)[cc]{\scriptsize$K_{r}$}}
\put(38.096,4.597){\makebox(0,0)[cc]{\scriptsize$K_{r}$}}
\put(47.819,.089){\makebox(0,0)[cc]{\footnotesize$G_{2r}$}}
\put(2.033,16.175){\makebox(0,0)[cc]{\scriptsize$v_3$}}
\put(10.077,16.264){\makebox(0,0)[cc]{\scriptsize$v_2$}}
\put(1.68,6.453){\makebox(0,0)[cc]{\scriptsize$w_1$}}
\put(10.342,6.364){\makebox(0,0)[cc]{\scriptsize$w_2$}}
\put(20.683,16.087){\makebox(0,0)[cc]{\scriptsize$v_3$}}
\put(28.815,16.264){\makebox(0,0)[cc]{\scriptsize$v_2$}}
\put(54.006,18.827){\makebox(0,0)[cc]{\scriptsize$v_{r}$}}
\put(54.094,3.536){\makebox(0,0)[cc]{\scriptsize$w_{r}$}}
\put(48.393,19.136){\oval(17.059,5.922)[]}
\put(48.481,5.082){\oval(17.059,5.922)[]}
\put(75.307,0){\makebox(0,0)[cc]{\footnotesize$G_{2r+1}$}}
\put(69.149,18.324){\circle*{.94}}
\put(79.225,18.501){\circle*{.94}}
\put(71.447,18.409){\circle*{.94}}
\put(69.149,4.889){\circle*{.94}}
\put(75.079,18.296){\circle*{.25}}
\put(72.781,4.86){\circle*{.25}}
\put(75.646,4.808){\circle*{.25}}
\put(77.369,4.889){\circle*{.94}}
%\emline(79.197,18.65)(71.507,4.95)
\multiput(79.197,18.65)(-.03372807,-.060087719){228}{\line(0,-1){.060087719}}
%\end
\put(73.48,18.409){\circle*{.94}}
\put(76.317,18.244){\circle*{.25}}
\put(79.225,4.889){\circle*{.94}}
\put(77.466,18.244){\circle*{.25}}
\put(74.196,4.808){\circle*{.25}}
\put(79.197,18.561){\line(0,-1){13.612}}
%\emline(79.197,5.037)(73.54,18.296)
\multiput(79.197,5.037)(-.033672619,.078922619){168}{\line(0,1){.078922619}}
%\end
%\emline(77.341,4.95)(73.363,18.386)
\multiput(77.341,4.95)(-.033711864,.113864407){118}{\line(0,1){.113864407}}
%\end
%\emline(71.33,18.473)(79.197,4.86)
\multiput(71.33,18.473)(.033619658,-.058175214){234}{\line(0,-1){.058175214}}
%\end
%\emline(79.109,18.561)(77.253,5.127)
\multiput(79.109,18.561)(-.03314286,-.23989286){56}{\line(0,-1){.23989286}}
%\end
\put(68.856,3.712){\makebox(0,0)[cc]{\scriptsize$w_1$}}
\put(71.507,3.712){\makebox(0,0)[cc]{\scriptsize$w_2$}}
\put(71.419,19.443){\makebox(0,0)[cc]{\scriptsize$v_2$}}
\put(68.944,19.357){\makebox(0,0)[cc]{\scriptsize$v_1$}}
\put(71.712,4.889){\circle*{.94}}
%\emline(79.108,18.65)(69.12,4.95)
\multiput(79.108,18.65)(-.0336296296,-.0461279461){297}{\line(0,-1){.0461279461}}
%\end
\put(63.906,17.501){\makebox(0,0)[cc]{\scriptsize$K_{r+1}$}}
\put(63.817,4.596){\makebox(0,0)[cc]{\scriptsize$K_{r}$}}
\put(81.937,19.092){\makebox(0,0)[cc]{\scriptsize$v_{r+1}$}}
\put(82.025,3.624){\makebox(0,0)[cc]{\scriptsize$w_{r}$}}
\put(75.882,19.136){\oval(17.059,5.922)[]}
\put(75.971,5.082){\oval(17.059,5.922)[]}
\end{picture}
\end{center}
\caption{The graphs $G_5$, $G_6$, $G_{2r}$ and $G_{2r+1}.$}\label{fig-1}
\end{figure}

%\begin{remark}\label{re-1}
%It is clear that $G_n$ is an induced subgraph of $G_{n+1}$ for $n\geq 2$, and
%\end{remark}

As in \cite{F.Duan}, we introduce the following three  graph transformations.
\begin{enumerate}[(I)]
\item A vertex $u$ of a graph is called a \emph{congruent vertex of I-type} if there exists a vertex $v\not\sim u$ such that  $N(u)=N(v)$. The graph transformation of deleting or adding a congruent vertex of I-type is called the \emph{(graph) transformation of I-type}.

\item A vertex $u$ of a graph is called a \emph{congruent vertex of II-type} if there exist two non-adjacent vertices $v$ and $w$ such that $N(u)$ is a disjoint union of $N(v)$ and $N(w)$. The graph transformation of deleting or adding a congruent vertex of II-type is called the \emph{(graph) transformation of II-type}.

\item An induced quadrangle $C_4=uvxy$ of a graph is called \emph{congruent} if it contains a pair of independent edges, say $uv$ and $xy$, such
that $N(u)\setminus \{y, v\}=N(v)\setminus \{u, x\}$ and $N(x)\setminus \{v, y\}=N(y)\setminus \{x, u\}$, where $uv$ and $xy$ make a pair of \emph{congruent edges} of $C_4$. A vertex is called a \emph{congruent vertex of III-type} if it belongs to some congruent quadrangle. The graph transformation of deleting or adding a congruent vertex of III-type is called the \emph{(graph) transformation of III-type}.
\end{enumerate}

\begin{figure}[t]
\begin{center}
\unitlength 1.2mm % = 2.845pt
\linethickness{0.4pt}
\ifx\plotpoint\undefined\newsavebox{\plotpoint}\fi % GNUPLOT compatibility
\setlength{\unitlength}{3.5pt}
\begin{picture}(88.189,15.556)(0,0)
\put(.536,14.726){\circle*{1.072}}
\put(9.47,14.726){\circle*{1.072}}
\put(.536,3.794){\circle*{1.072}}
\put(9.47,3.794){\circle*{1.072}}
\put(.42,14.82){\line(1,0){9.145}}
\put(.42,3.889){\line(1,0){9.04}}
\put(9.46,14.82){\line(0,-1){10.932}}
\put(.536,10.731){\circle*{1.072}}
%\emline(.526,10.615)(9.565,14.715)
\multiput(.526,10.615)(.074090164,.033606557){122}{\line(1,0){.074090164}}
%\end
\put(26.288,14.831){\circle*{1.072}}
\put(35.222,14.831){\circle*{1.072}}
\put(26.288,3.899){\circle*{1.072}}
\put(35.222,3.899){\circle*{1.072}}
\put(26.172,14.925){\line(1,0){9.145}}
\put(26.172,3.994){\line(1,0){9.04}}
\put(35.212,14.925){\line(0,-1){10.932}}
\put(26.288,10.836){\circle*{1.072}}
%\emline(26.278,10.72)(35.317,14.82)
\multiput(26.278,10.72)(.074090164,.033606557){122}{\line(1,0){.074090164}}
%\end
\put(32.7,6.317){\circle*{1.072}}
%\emline(35.213,14.926)(32.69,6.306)
\multiput(35.213,14.926)(-.03364,-.11493333){75}{\line(0,-1){.11493333}}
%\end
%\emline(26.173,3.784)(32.69,6.306)
\multiput(26.173,3.784)(.08689333,.03362667){75}{\line(1,0){.08689333}}
%\end
\put(51.515,14.936){\circle*{1.072}}
\put(60.449,14.936){\circle*{1.072}}
\put(51.515,4.004){\circle*{1.072}}
\put(60.449,4.004){\circle*{1.072}}
\put(51.399,15.03){\line(1,0){9.145}}
\put(51.399,4.099){\line(1,0){9.04}}
\put(60.439,15.03){\line(0,-1){10.932}}
\put(51.515,10.941){\circle*{1.072}}
%\emline(51.505,10.825)(60.544,14.925)
\multiput(51.505,10.825)(.074090164,.033606557){122}{\line(1,0){.074090164}}
%\end
\put(55.93,8.734){\circle*{1.072}}
%\emline(55.814,8.724)(60.544,15.136)
\multiput(55.814,8.724)(.033546099,.045475177){141}{\line(0,1){.045475177}}
%\end
%\emline(55.709,8.724)(60.544,4.099)
\multiput(55.709,8.724)(.035036232,-.033514493){138}{\line(1,0){.035036232}}
%\end
\put(77.373,14.936){\circle*{1.072}}
\put(86.307,14.936){\circle*{1.072}}
\put(77.373,4.004){\circle*{1.072}}
\put(86.307,4.004){\circle*{1.072}}
\put(77.257,15.03){\line(1,0){9.145}}
\put(77.257,4.099){\line(1,0){9.04}}
\put(86.297,15.03){\line(0,-1){10.932}}
\put(77.373,10.941){\circle*{1.072}}
%\emline(77.363,10.825)(86.402,14.925)
\multiput(77.363,10.825)(.074090164,.033606557){122}{\line(1,0){.074090164}}
%\end
\put(80.841,9.47){\circle*{1.072}}
\put(80.831,9.565){\line(-2,-3){3.574}}
%\emline(80.726,9.565)(86.507,15.031)
\multiput(80.726,9.565)(.035466258,.033533742){163}{\line(1,0){.035466258}}
%\end
%\emline(80.936,9.46)(77.152,10.931)
\multiput(80.936,9.46)(-.086,.03343182){44}{\line(-1,0){.086}}
%\end
%\emline(80.936,9.565)(77.257,15.241)
\multiput(80.936,9.565)(-.033445455,.0516){110}{\line(0,1){.0516}}
%\end
\put(4.415,0){\makebox(0,0)[cc]{\footnotesize$H\in \mathcal{T}_5^1$}}
\put(30.272,.105){\makebox(0,0)[cc]{\footnotesize$H_1\in \mathcal{T}_6^2(I)$}}
\put(81.777,.63){\makebox(0,0)[cc]{\footnotesize$H_3\in \mathcal{T}_6^2(III)$}}
\put(55.499,.21){\makebox(0,0)[cc]{\footnotesize$H_2\in \mathcal{T}_6^2(II)$}}
\put(31.323,7.568){\makebox(0,0)[cc]{\scriptsize$u$}}
\put(36.894,4.099){\makebox(0,0)[cc]{\scriptsize$v$}}
\put(49.823,15.556){\makebox(0,0)[cc]{\scriptsize$v$}}
\put(49.718,5.045){\makebox(0,0)[cc]{\scriptsize$w$}}
\put(54.448,7.673){\makebox(0,0)[cc]{\scriptsize$u$}}
\put(82.197,8.724){\makebox(0,0)[cc]{\scriptsize$u$}}
\put(88.189,15.136){\makebox(0,0)[cc]{\scriptsize$v$}}
\put(88.189,4.625){\makebox(0,0)[cc]{\scriptsize$x$}}
\put(75.996,4.73){\makebox(0,0)[cc]{\scriptsize$y$}}
\end{picture}
\end{center}
\caption{The graphs $H$, $H_1$, $H_2$ and $H_3$ (for Example~\ref{exa:H}).}\label{fig-2}
\end{figure}

\begin{exa}\label{exa:H}Figure~\ref{fig-2} illustrates the previous graph transformations. There, we take a graph~$H$, and then we obtain $H_1$ (resp.~$H_2$, $H_3$) by adding a congruent vertex $u$ of I-type (resp.~II-type, III-type).
	
	It is worth mentioning that, apart from $u$, the vertices $x,y,v$ are also congruent of III-type in $H_3$. By deleting one of them we obtain a graph which is not necessarily isomorphic to $H$, but its nullity decreases by $1$ and its positive and negative inertia indices remain unchanged.
\end{exa}

	The previous example illustrates the following result of~\cite{F.Duan}.

\begin{lem}\cite{F.Duan}\label{lem-2-4}
If $u$ is a congruent vertex of type I, II or III of $G$, then $p(G)=p(G-u)$, $n(G)=n(G-u)$ and $\eta(G)=\eta(G-u)+1$.
\end{lem}

Finally, we define the fourth graph transformation as follows.

\begin{enumerate}[(IV)]
	\item A vertex $u$ is called a \emph{congruent vertex of IV-type} in a graph $G$ if it is an isolated vertex. The graph transformation of deleting or adding a congruent vertex of IV-type is called the \emph{(graph) transformation of IV-type}. Clearly, we have $p(G)=p(G-u)$, $n(G)=n(G-u)$ and $\eta(G)=\eta(G-u)+1$.
\end{enumerate}

%Recall that $\mathcal{T}_n=\mathcal{T}^0_n\cup \mathcal{T}^1_n\cup \cdots \cup \mathcal{T}^{n-3}_n$. In Section 4, we shall use the  above four graph transformations to characterize the the graphs in $\mathcal{T}_n^s$ for $s\in \{2, \ldots,n-3\}$.

\section{Characterization of graphs of $\mathcal{T}_n^0$ and  $\mathcal{T}_n^1$}\label{sec:3}

Oboudi \cite{M.R.Oboudi3} has determined all graphs belonging to $\mathcal{T}_n^0$.

\begin{thm}\cite{M.R.Oboudi3}\label{thm-2-3}
For a graph $G\in \mathcal{T}_n^0$ the following holds true.
\begin{enumerate}[(i)]
\item If $G$ is disconnected, then $G\cong K_p + K_q$ for some integers $p, q\geq 2$;
\item If $G$ is connected, then there exist positive integers $k$ and $n_1, n_2, \ldots, n_k$ such that $3\leq k\leq 12$, $n_1+n_2 + \cdots + n_k=n$ and $G\cong G_k[K_{n_1}, K_{n_2}, \ldots, K_{n_k}]$.
\end{enumerate}
\end{thm}

Moreover, Oboudi determined the mentioned integers for which $G_k[K_{n_1}, K_{n_2}, \ldots, K_{n_k}]\in \mathcal{T}_n^0$. Accordingly, the set of connected graphs belonging to $\mathcal{T}_n^0$ consists of exactly 48 infinite families and additional 601 individual graphs listed in~\cite{T.Derikvand}. All of them are reviewed Table~\ref{tab-2-1}.

\begin{table}
{\scriptsize
	\captionsetup{font=scriptsize}
\caption{{\scriptsize The 48 infinite families and the 601 individual graphs $G_k[K_{n_1}, K_{n_2}, \ldots, K_{n_k}]$ of $\mathcal{T}_n^0$}.}
\centering
\begin{tabular}{ccc}
\toprule
$k$ &$\#$ infinite families &$\#$ individual graphs  \\
\midrule
\makecell[t]{ 3 } &
\makecell[t]{ 1  (\cite[Theorem~3.4]{M.R.Oboudi3})} &

\makecell[t]{0}\\

\makecell[t]{4 } &
\makecell[t]{ 8 (\cite[Theorem~3.5]{M.R.Oboudi3}) } &
\makecell[t]{ 25 (\cite[Theorem~3.5]{M.R.Oboudi3}(3)) } \\

\makecell[t]{5 } &
\makecell[t]{15  (\cite[Theorem~3.6]{M.R.Oboudi3})} &
\makecell[t]{ 63 (\cite[Theorem~3.6]{M.R.Oboudi3}(5))} \\

\makecell[t]{6 } &
\makecell[t]{ 13 (\cite[Theorem~3.7]{M.R.Oboudi3})} &
\makecell[t]{ 145 (\cite[Theorem~3.7]{M.R.Oboudi3}(5))} \\

\makecell[t]{7 } &
\makecell[t]{ 8 (\cite[Theorem~3.8]{M.R.Oboudi3})} &
\makecell[t]{ 143 (\cite[Theorem~3.8]{M.R.Oboudi3}(4))} \\

\makecell[t]{8 } &
\makecell[t]{ 2 (\cite[Theorem~3.9]{M.R.Oboudi3})} &
\makecell[t]{ 134 (\cite[Theorem~3.9]{M.R.Oboudi3}(2))} \\

\makecell[t]{9 } &
\makecell[t]{ 1 (\cite[Theorem~3.10]{M.R.Oboudi3})} &
\makecell[t]{ 59 (\cite[Theorem~3.5]{M.R.Oboudi3}(2))} \\

\makecell[t]{10} &
\makecell[t]{ 0} &
\makecell[t]{ 26  (\cite[Theorem 3.12]{M.R.Oboudi3})} \\

\makecell[t]{11} &
\makecell[t]{0} &
\makecell[t]{ 5  (\cite[Theorem 3.13]{M.R.Oboudi3})} \\

\makecell[t]{12} &
\makecell[t]{0} &
\makecell[t]{ 1 (\cite[Theorem 3.14]{M.R.Oboudi3})} \\
\bottomrule
\end{tabular}\label{tab-2-1}}
\end{table}
We say that a graph $G\in \mathcal{T}_n$ is \emph{derivable} if there exists $G'\in \mathcal{T}_{n-1}$ such that $G$ is obtained from $G'$ by adding a congruent vertex of type I--IV. For otherwise, we say that $G$ \emph{underivable}. By definition, graphs of $\mathcal{T}_n^0$ are  underivable. In what follows, for $\ast\in\{I, II, III, IV\}$, we write $\mathcal{T}_n^s(\ast)$ to denote the collection of graphs of $\mathcal{T}_n^s$ obtained by adding a congruent vertex of $\ast$-type to a graph of $\mathcal{T}_{n-1}^{s-1}$.

 We proceed with characterization of graphs belonging to $\mathcal{T}_n^1$.

\begin{thm}\cite{F.Duan}\label{thm-3-1}
Let $G$ be a graph of order $n\ge 5$. Then $G\in \mathcal{T}_{n}^1$ if and only if $G$ is isomorphic to one of the following graphs:
\begin{enumerate}[{\rm (i)}]
\item $K_s+K_t+K_1$, $K_s+K_{n-s}- e$ for $e\in E(K_{n-s})$ (where $s,t\ge 2$, $s+t=n-1$) and $H+K_1$ where $H\in \mathcal{T}_{n-1}^0$ is connected;
\item $K_{1, 2}(u)\odot^k K_{n-3}$ and $K_{1, 1}(u)\odot^k (K_{n-2}- e)$ for $e\in E(K_{n-2})$, where $u$ is a vertex of the maximum degree in the first graph;
\item the graphs belonging to $\mathcal{T}_n^1(I)$, $\mathcal{T}_n^1(II)$ or $\mathcal{T}_n^1(III)$;
\item the $802$ individual graphs listed in \cite[Tables 1--6]{F.Duan}.
\end{enumerate}
\end{thm}

We review the graphs of the item (iv) in Table~\ref{tab-3-2}

\begin{table}
	\scriptsize
		\captionsetup{font=scriptsize}
	\caption{\scriptsize{The 802 individual graphs $G_k[K_{n_1}, K_{n_1}, \ldots, K_{n_k}]$ of $\mathcal{T}_n^1$.}}
	\centering
	\begin{tabular}{cccc}
		\toprule
		$k$  & $\#$ graphs & $k$  & $\#$ graphs\\
		\midrule
		\makecell[t]{ 4 } &
		\makecell[t]{ 18 (\cite[Table~1]{F.Duan})} & \makecell[t]{ 9 } & \makecell[t]{ 124 (\cite[Table~5]{F.Duan})} \\
		
		\midrule
		\makecell[t]{ 5 } & \makecell[t]{ 47 (\cite[Table~1]{F.Duan})} & \makecell[t]{ 10 } & \makecell[t]{ 78 (\cite[Table~6]{F.Duan})} \\
		\midrule
		\makecell[t]{ 6 } & \makecell[t]{ 138 (\cite[Table~2]{F.Duan})} & \makecell[t]{ 11 } & \makecell[t]{ 24 (\cite[Table~1]{F.Duan})} \\
		\midrule
		\makecell[t]{ 7 } & \makecell[t]{ 161 (\cite[Table~3]{F.Duan})} & \makecell[t]{ 12 } & \makecell[t]{ 6 (\cite[Table~1]{F.Duan})} \\
		\midrule
		\makecell[t]{ 8 } & \makecell[t]{ 205 (\cite[Table~4]{F.Duan})} & \makecell[t]{ 13 } & \makecell[t]{ 1 (\cite[Table~1]{F.Duan})} \\
		\bottomrule
	\end{tabular}\label{tab-3-2}
\end{table}

In order to rephrase the result of Theorem~\ref{thm-3-1}, we prove that each of  graphs described in the items (i) and (ii) belongs to $\mathcal{T}_n^1(I)$, $\mathcal{T}_n^1(II)$ or $\mathcal{T}_n^1(IV)$.

\begin{thm} We have
\begin{enumerate}[{\rm (i)}]
\item $K_s+K_{n-s}- e, K_{1, 2}(u)\odot^k K_{n-3}\in \mathcal{T}_n^1(I)$, where $e\in E(K_{n-s})$ and $1\le k\le n-3$;
\item $K_{1, 1}(u)\odot^k (K_{n-2}- e)\in \mathcal{T}_n^1(I)$ or $\mathcal{T}_n^1(II)$, where $1\le k\le n-2$;
\item $K_s+K_t+K_1, H+K_1\in \mathcal{T}_n^1(IV)$, where $H\in \mathcal{T}_{n-1}^0$ is connected and $s,t\ge 2$.
\end{enumerate}
\end{thm}
\begin{proof}
(i): Let $H_1\cong K_s+K_{n-s}- e$ where $e=u_1u_2$ is any  edge of $K_{n-s}$. Then $N_{H_1}(u_1)=N_{H_1}(u_2)=V(K_{n-s})\setminus\{u_1, u_2\}$ and so $u_1$ is a congruent vertex of I-type. Since $H_1\in \mathcal{T}_n^1$ (by Theorem \ref{thm-3-1}(i)), we have $H_1- u_1\in \mathcal{T}_{n-1}^0$, and thus $H_1\in \mathcal{T}_n^1(I)$. The graph $H_2\cong K_{1, 2}(u)\odot^k K_{n-3}$  has two pendant vertices $v_1$ and $v_2$ such that $N_{H_2}(v_1)=N_{H_2}(v_2) =\{u\}$. Thus $H_2- v_1$ belongs to $\mathcal{T}_{n-1}^0$  by Theorem~\ref{thm-3-1}(ii), and so $H_2\in \mathcal{T}_n^1(I)$.

(ii): Let $H_3=K_{1, 1}(u)\odot^k (K_{n-2}- e)$, where $e=w_1w_2$ is any edge of $K_{n-2}$, and let $v$ be the vertex of $K_{1, 1}$ other than $u$. If $w_1, w_2\in N_{H_3}(u)$, then $N_{H_3}(w_1)=N_{H_3}(w_2)=V(H_3)\setminus \{w_1, w_2, v\}$. If $w_1, w_2\notin N_{H_3}(u)$, then $N_{H_3}(w_1)= N_{H_3}(w_2)=V(H_3)\setminus \{w_1, w_2,u,v\}$. Thus $w_1$ and $w_2$ are congruent vertices of I-type in the both  cases, and so $H_3- w_1$ belongs to $\mathcal{T}_{n-1}^0$ by Theorem \ref{thm-3-1}(ii). Hence, $H_3\in \mathcal{T}_n^1(I)$. At last, assume that exactly one of $w_1$ and $w_2$, say $w_1$, belongs to $N_{H_3}(u)$. We see that $N_{H_3}(w_1)$ is a disjoint union of $N_{H_3}(w_2)$ and $N_{H_3}(v)$.  Thus $w_1$ is a congruent vertex of II-type and $H_3- w_1$ belongs to $\mathcal{T}_{n-1}^0$. Hence, $H_3\in \mathcal{T}_n^1(II)$.

(iii): It is clear that $K_s+K_t+K_1, H+K_1\in \mathcal{T}_n^1(IV)$ since $K_s+K_t,H\in \mathcal{T}_{n-1}^0$.
\end{proof}

In the light of the last two theorems, graphs of $\mathcal{T}_n^1$ can be partitioned into two classes:  derivable graphs included in the union of $\mathcal{T}_n^1(I)$, $\mathcal{T}_n^1(II)$, $\mathcal{T}_n^1(III)$ and $\mathcal{T}_n^1(IV)$, and the 802 underivable graphs described in Theorem~\ref{thm-3-1}(iv) and reviewed in Table~\ref{tab-3-2}.

\section{Characterization of graphs of $\mathcal{T}_n^s$, for $2\leq s\leq n-3$}\label{sec:4}

\begin{figure}
	\centering
	\includegraphics[width=100mm,angle=0]{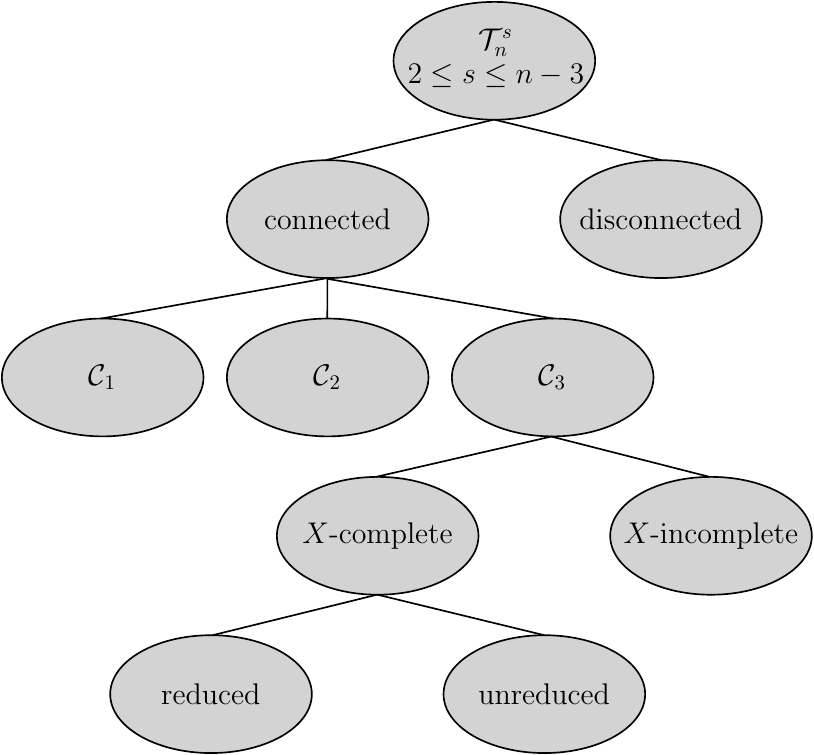}
	\caption{A sketched concept of the proof of Theorem~\ref{thm-1-2}.}\label{fig:concept}
\end{figure}

In this section we prove Theorem~\ref{thm-1-2}. Our considerably long proof is divided into a sequence of lemmas. To make reading the text easier, in Figure~\ref{fig:concept} we give a sketch of the proof. Accordingly, we partition the graphs of $\mathcal{T}_n^s,~2\leq s\leq n-3$, into connected ones and disconnected ones. Connected ones are further partitioned into three classes (named $\mathcal{C}_1-\mathcal{C}_3$ and defined in the corresponding part of this section). The first two classes are considered directly, while the third one is further partitioned into two subclasses, and the first of them is again partitioned in a similar way.

We first single out the case in which $G$ is disconnected, i.e.~we prove  part (a) of Theorem~\ref{thm-1-2}.

\medskip

{\it \bf Proof of Theorem \ref{thm-1-2}(a).} A direct computation shows that all graphs listed in Theorem~\ref{thm-1-2}(a) have exactly two positive eigenvalues.

We proceed with the necessity. Let $G\in \mathcal{T}_n^s, 2\leq s\leq n-3$, be disconnected with components $H_1, H_2,\ldots,H_k$ ($k\geq 2$). Since $p(G)=2$, $G$ has at most two non-trivial components. If there is exactly one non-trivial component, say $H_1$, then $G\cong H_1+(k-1)K_1$. It is clear that $2=p(G)=p(H_1)$ and $s=\eta(G)=\eta(H_1)+(k-1)$. Thus  $H_1\in \mathcal{T}^{s-l}_{n-l}$ where $l=k-1$.

If there are two non-trivial components, say $H_1$ and $H_2$, then $G\cong H_1+H_2+(k-2)K_1$. Therefore, $n=|H_1|+|H_2|+(k-2)$,  $p(H_1)=1=p(H_2)$ and   $s=\eta(G)=\eta(H_1)+\eta(H_2)+k-2$. From Theorem~\ref{thm-2-2} we have $H_1\cong K_{t_1, t_2, \ldots, t_p}$ and $H_2\cong K_{r_1, r_2, \ldots, r_q}$. Note
that $\eta(K_{t_1, t_2, \ldots, t_p})=(t_1-1)+\cdots +(t_p-1)$ and $\eta(K_{r_1, r_2, \ldots, r_q})=(r_1-1)+\cdots +(r_q-1)$. By taking $m=k-2\geq 0$, we obtain $s=(t_1-1)+(t_2-1)+\cdots+(t_p-1)+(r_1-1)+(r_2-1)+\cdots +(r_q-1)+m$, and we are done.\qed

\begin{remark}\label{re-4-1}
If $G\cong H+lK_1$, where $H\in \mathcal{T}^{s-l}_{n-l}$ is connected and $l\geq 1$, then $G$ clearly belongs to $\mathcal{T}_n^s(IV)$.
Let $G\cong K_{t_1, t_2, \ldots, t_p}+K_{r_1, r_2, \ldots, r_q}+mK_1$, where $m\geq 0$. If $m\geq 1$, then $G$ also belongs to $\mathcal{T}_n^s(IV)$. If $m=0$, then at least one of $t_1, t_2, \ldots, t_p, r_1, r_2, \ldots, r_q$ is greater than or equal to~2. Therefore, there exist two non-adjacent vertices $u, v\in V(G)$ such that $N(u)=N(v)$, and so $u$ is a congruent vertex of I-type, which means that in this case we have $G\in \mathcal{T}_n^s(I)$.
\end{remark}

In what follows we assume that $G$ is a connected graph of $\mathcal{T}_n^s, 2\leq s\leq n-3$. Since $s\geq 2$, we have $\lambda_1(G)>\lambda_2(G)>\lambda_3(G)=\lambda_4(G)=0$.
Starting from this point we reserve the symbol $v^*$ to denote a vertex of $G$ with minimum vertex degree, say $d(v^*)=t$, and set $X=N_G(v^*), Y=V(G)\setminus N_G[v^*]$. Clearly, $Y\not=\emptyset$, since for otherwise $G$ would be a complete graph (with $\lambda_2(G)<0$). The following result determines the induced subgraph $G[Y]$.

\begin{lem}\label{lem-4-1}
We have $G[Y]\cong (n-t-1)K_1$ or $G[Y]\cong K_{n_1, n_2, \ldots, n_l}+rK_1$ where $n_1+n_2+\cdots + n_l+r=n-t-1$ and $r\ge 0$.
\end{lem}
\renewcommand\proofname{\it Proof}
\begin{proof}
For every $x\in X$, the induced subgraph $G[\{v^*, x\}\cup Y]$ has a pendant vertex $v^*$ by the choice of $v^*$. By Lemmas~\ref{lem-2-1} and~\ref{lem-2-2}, we have $2=p(G)\geq p(G[\{v^*, x\}\cup Y])=p(G[Y])+1$. Hence, $p(G[Y])\leq 1$.

If $p(G[Y])=0$, then $G[Y]\cong (n-t-1)K_1$. If $p(G[Y])=1$, then by Theorem \ref{thm-2-2} we have $G[Y]\cong K_{n_1, n_2, \ldots, n_l}+ rK_1$,
where $n_1+ n_2+\cdots n_l+r=n-t-1$ and $r\ge 0$. This completes the proof.
\end{proof}

By virtue of Lemma~\ref{lem-4-1}  we can partition connected graphs of $\mathcal{T}_n^s~ (2\leq s\leq n-3)$ into the following three classes determined by $G[Y]$.
\smallskip
\begin{itemize}
\item $\mathcal{C}_1$: $G[Y]\cong (n-t-1)K_1$ or $G[Y]\cong K_{n_1, n_2, \ldots, n_l}+rK_1$,  where $n_1+\cdots +n_l+r=n-t-1$ and $r\ge 1$;
\item $\mathcal{C}_2$:  $G[Y]\cong K_{n_1, n_2, \ldots, n_l}$, where $n_1+ n_2+\cdots +n_l=n-t-1$ and there exists some $i\in \{1,2,\ldots,l\}$ such that $n_i\geq 2$;
\item $\mathcal{C}_3$:  $G[Y]\cong K_{n_1, n_2, \ldots, n_l}$, where $n_1+ n_2+\cdots +n_l=n-t-1$ and $n_1=n_2=\cdots=n_l=1$, that is $G[Y]\cong K_{n-t-1}$, where $n-t-1\geq 2$.
\end{itemize}
\smallskip

We will see that graphs of $\mathcal{C}_1$ and $\mathcal{C}_2$ belong to $\mathcal{T}^s_{n}(I)\cup \mathcal{T}^{s}_{n}(II)\cup \mathcal{T}^{s}_{n}(III)$ and that there is a finite number of graphs of $\mathcal{C}_3$ that do not belong to the same union. We will obtain all of them till the end of this section. Precisely, the class $\mathcal{C}_1$ is considered in the forthcoming Lemma~\ref{lem-4-2}, the class $\mathcal{C}_2$ in Lemma~\ref{lem-4-3} and the class $\mathcal{C}_3$ in Lemmas~\ref{lem-4-4}, \ref{lem-4-5} and~\ref{lem-4-13}.

\begin{lem}\label{lem-4-2}
If $G[Y]\cong (n-t-1)K_1$  or $G[Y]\cong K_{n_1, n_2, \ldots, n_l}+rK_1$ where $r\ge 1$, then $G\in \mathcal{T}_{n}^{s}(I)$.
\end{lem}
\begin{proof}
By the choice of $v^*$, in both cases $G[Y]$ has an isolated vertex, say $y$. We have  $N_G(y)=X=N(v^*)$ since $G$ is connected and  $d(y)\geq t$. Thus $y$ is a congruent vertex (with respect to $v^*$) of I-type. By Lemma \ref{lem-2-4}, we have $p(G)=p(G-y)$ and $\eta(G)=\eta(G-y)+1$. Thus $G-y\in \mathcal{T}^{s-1}_{n-1}$, and so $G\in \mathcal{T}_{n}^{s}(I)$.
\end{proof}

\begin{figure}[t]
	\begin{center}
		\unitlength 1.2mm % = 2.845pt
		\linethickness{0.5pt}
		\ifx\plotpoint\undefined\newsavebox{\plotpoint}\fi % GNUPLOT compatibility
		\setlength{\unitlength}{3.5pt}
		\begin{picture}(117.25,52.125)(0,0)
			\put(5.958,50.092){\line(1,-1){3.889}}
			\put(1.716,46.203){\line(0,-1){4.861}}
			\put(9.936,46.114){\line(0,-1){4.95}}
			\put(6.049,49.928){\circle*{1.226}}
			\put(10.062,40.86){\circle*{1.226}}
			\put(10.062,45.914){\circle*{1.226}}
			\put(1.738,45.766){\circle*{1.226}}
			\put(1.887,40.712){\circle*{1.226}}
			\put(6.049,36.995){\circle*{1.226}}
			%\emline(1.719,40.625)(6.094,37)
			\multiput(1.719,40.625)(.040509259,-.033564815){108}{\line(1,0){.040509259}}
			%\end
			\put(6.125,51.525){\makebox(0,0)[cc]{\scriptsize$v^\ast$}}
			%\emline(5.875,37)(10.125,40.875)
			\multiput(5.875,37)(.036956522,.033695652){115}{\line(1,0){.036956522}}
			%\end
			%\emline(1.625,45.875)(6.125,50.125)
			\multiput(1.625,45.875)(.035714286,.033730159){126}{\line(1,0){.035714286}}
			%\end
			\put(11.625,40.625){\makebox(0,0)[cc]{\scriptsize$y^\prime$}}
			\put(11.5,46.125){\makebox(0,0)[cc]{\scriptsize$x^\prime$}}
			\put(8.75,36){\makebox(0,0)[cc]{\scriptsize$y^\ast$}}
			\put(.125,45.625){\makebox(0,0)[cc]{\scriptsize$x$}}
			\put(0,40){\makebox(0,0)[cc]{\scriptsize$y$}}
			\put(6,32.75){\makebox(0,0)[cc]{\scriptsize$\lambda_3(C_6)=1$}}
			\put(6.125,29.5){\makebox(0,0)[cc]{\footnotesize$C_6$}}
			%\emline(6.265,20.155)(2.023,16.089)
			\multiput(6.265,20.155)(-.035057851,-.033603306){121}{\line(-1,0){.035057851}}
			%\end
			\put(6.265,20.067){\line(1,-1){3.889}}
			\put(10.243,16.089){\line(0,-1){4.95}}
			%\emline(2.023,11.228)(6.354,7.162)
			\multiput(2.023,11.228)(.035793388,-.033603306){121}{\line(1,0){.035793388}}
			%\end
			%\emline(2.111,16.178)(10.331,11.051)
			\multiput(2.111,16.178)(.054078947,-.033730263){152}{\line(1,0){.054078947}}
			%\end
			\put(6.276,19.824){\circle*{1.226}}
			\put(10.14,15.959){\circle*{1.226}}
			\put(10.289,10.757){\circle*{1.226}}
			\put(6.276,7.338){\circle*{1.226}}
			\put(2.262,15.959){\circle*{1.226}}
			%\emline(6.125,7.375)(10.375,10.875)
			\multiput(6.125,7.375)(.040865385,.033653846){104}{\line(1,0){.040865385}}
			%\end
			\put(2,16.125){\line(1,0){8.125}}
			\put(2.125,15.875){\line(1,-2){4.25}}
			\put(2.238,10.988){\circle*{1.226}}
			\put(2.25,11){\line(0,1){5.125}}
			\put(7.25,.125){\makebox(0,0)[cc]{\footnotesize$\Gamma_6$}}
			\put(6.125,21.5){\makebox(0,0)[cc]{\scriptsize$v^\ast$}}
			\put(.5,16){\makebox(0,0)[cc]{\scriptsize$x$}}
			\put(.5,11){\makebox(0,0)[cc]{\scriptsize$y$}}
			\put(12,16.375){\makebox(0,0)[cc]{\scriptsize$x^\prime$}}
			\put(12.125,10.625){\makebox(0,0)[cc]{\scriptsize$y^\prime$}}
			\put(8.875,6.5){\makebox(0,0)[cc]{\scriptsize$y^\ast$}}
			\put(6.75,3){\makebox(0,0)[cc]{\scriptsize$\lambda_3(\Gamma_6)=0.1124$}}
			\put(25.833,49.967){\line(1,-1){3.889}}
			\put(21.591,46.078){\line(0,-1){4.861}}
			\put(29.811,45.989){\line(0,-1){4.95}}
			\put(25.924,49.803){\circle*{1.226}}
			\put(29.937,40.735){\circle*{1.226}}
			\put(29.937,45.789){\circle*{1.226}}
			\put(21.613,45.641){\circle*{1.226}}
			\put(21.762,40.587){\circle*{1.226}}
			\put(25.924,36.87){\circle*{1.226}}
			\put(21.469,45.625){\line(1,0){8.375}}
			%\emline(21.594,40.5)(25.969,36.875)
			\multiput(21.594,40.5)(.040509259,-.033564815){108}{\line(1,0){.040509259}}
			%\end
			\put(26,51.5){\makebox(0,0)[cc]{\scriptsize$v^\ast$}}
			%\emline(25.75,36.875)(30,40.75)
			\multiput(25.75,36.875)(.036956522,.033695652){115}{\line(1,0){.036956522}}
			%\end
			%\emline(21.5,45.75)(26,50)
			\multiput(21.5,45.75)(.035714286,.033730159){126}{\line(1,0){.035714286}}
			%\end
			\put(31.5,40.5){\makebox(0,0)[cc]{\scriptsize$y^\prime$}}
			\put(31.375,46){\makebox(0,0)[cc]{\scriptsize$x^\prime$}}
			\put(28.625,35.875){\makebox(0,0)[cc]{\scriptsize$y^\ast$}}
			\put(26.125,32.625){\makebox(0,0)[cc]{\scriptsize$\lambda_3(\Gamma_1)=0.6180$}}
			\put(20,45.5){\makebox(0,0)[cc]{\scriptsize$x$}}
			\put(19.875,39.875){\makebox(0,0)[cc]{\scriptsize$y$}}
			\put(26.25,29.5){\makebox(0,0)[cc]{\footnotesize$\Gamma_1$}}
			\put(46.487,49.568){\line(1,-1){3.889}}
			\put(50.465,45.59){\line(0,-1){4.95}}
			\put(50.376,40.64){\line(-1,-1){3.977}}
			\put(46.549,49.255){\circle*{1.226}}
			\put(50.414,45.39){\circle*{1.226}}
			\put(50.562,40.485){\circle*{1.226}}
			\put(46.549,36.769){\circle*{1.226}}
			\put(42.387,45.688){\circle*{1.226}}
			\put(42.363,40.363){\circle*{1.226}}
			\put(47.375,29.5){\makebox(0,0)[cc]{\footnotesize$\Gamma_2$}}
			\put(46.625,51.125){\makebox(0,0)[cc]{\scriptsize$v^\ast$}}
			\put(52.125,40){\makebox(0,0)[cc]{\scriptsize$y^\prime$}}
			\put(52,45.5){\makebox(0,0)[cc]{\scriptsize$x^\prime$}}
			\put(49.125,35.75){\makebox(0,0)[cc]{\scriptsize$y^\ast$}}
			\put(46.875,32.375){\makebox(0,0)[cc]{\scriptsize$\lambda_3(\Gamma_2)=0.4142$}}
			%\emline(42.25,45.75)(46.5,49.375)
			\multiput(42.25,45.75)(.039351852,.033564815){108}{\line(1,0){.039351852}}
			%\end
			\put(42.25,45.75){\line(0,-1){5.625}}
			%\emline(42.25,45.75)(50.5,40.5)
			\multiput(42.25,45.75)(.052884615,-.033653846){156}{\line(1,0){.052884615}}
			%\end
			%\emline(42.125,40.375)(46.5,36.75)
			\multiput(42.125,40.375)(.040509259,-.033564815){108}{\line(1,0){.040509259}}
			%\end
			\put(40.875,45.75){\makebox(0,0)[cc]{\scriptsize$x$}}
			\put(40.75,40.125){\makebox(0,0)[cc]{\scriptsize$y$}}
			%\emline(67.53,49.541)(63.288,45.475)
			\multiput(67.53,49.541)(-.035057851,-.033603306){121}{\line(-1,0){.035057851}}
			%\end
			\put(67.53,49.453){\line(1,-1){3.889}}
			\put(63.288,45.564){\line(0,-1){4.861}}
			\put(71.508,45.475){\line(0,-1){4.95}}
			\put(67.573,49.201){\circle*{1.226}}
			\put(67.424,36.863){\circle*{1.226}}
			\put(71.437,40.282){\circle*{1.226}}
			\put(63.113,40.431){\circle*{1.226}}
			\put(63.262,45.187){\circle*{1.226}}
			\put(63.219,45.219){\line(1,0){8.25}}
			\put(71.457,45.082){\circle*{1.226}}
			%\emline(67.469,36.969)(71.594,40.219)
			\multiput(67.469,36.969)(.042525773,.033505155){97}{\line(1,0){.042525773}}
			%\end
			\put(68.25,29.625){\makebox(0,0)[cc]{\footnotesize$\Gamma_3$}}
			\put(67.75,51.175){\makebox(0,0)[cc]{\scriptsize$v^\ast$}}
			\put(73.625,39.875){\makebox(0,0)[cc]{\scriptsize$y^\prime$}}
			\put(73.5,45.375){\makebox(0,0)[cc]{\scriptsize$x^\prime$}}
			\put(69.625,35.75){\makebox(0,0)[cc]{\scriptsize$y^\ast$}}
			\put(68,32.75){\makebox(0,0)[cc]{\scriptsize$\lambda_3(\Gamma_3)=0.1830$}}
			%\emline(63.125,40.5)(67.5,36.75)
			\multiput(63.125,40.5)(.0390625,-.033482143){112}{\line(1,0){.0390625}}
			%\end
			\put(61.375,45.125){\makebox(0,0)[cc]{\scriptsize$x$}}
			\put(61.375,40.125){\makebox(0,0)[cc]{\scriptsize$y$}}
			\put(91.25,29.5){\makebox(0,0)[cc]{\footnotesize$\Gamma_4$}}
			\put(89.75,32.5){\makebox(0,0)[cc]{\scriptsize$\lambda_3(\Gamma_4)=0.5293$}}
			%\emline(89.28,49.666)(85.038,45.6)
			\multiput(89.28,49.666)(-.035057851,-.033603306){121}{\line(-1,0){.035057851}}
			%\end
			\put(89.28,49.578){\line(1,-1){3.889}}
			\put(85.038,45.689){\line(0,-1){4.861}}
			\put(93.258,45.6){\line(0,-1){4.95}}
			\put(89.323,49.326){\circle*{1.226}}
			\put(89.174,36.988){\circle*{1.226}}
			\put(93.187,40.407){\circle*{1.226}}
			\put(84.863,40.556){\circle*{1.226}}
			\put(85.012,45.312){\circle*{1.226}}
			\put(84.969,45.344){\line(1,0){8.25}}
			\put(93.207,45.207){\circle*{1.226}}
			%\emline(84.844,45.344)(89.219,36.969)
			\multiput(84.844,45.344)(.033653846,-.064423077){130}{\line(0,-1){.064423077}}
			%\end
			%\emline(89.219,37.094)(93.344,40.344)
			\multiput(89.219,37.094)(.042525773,.033505155){97}{\line(1,0){.042525773}}
			%\end
			\put(89.5,51.5){\makebox(0,0)[cc]{\scriptsize$v^\ast$}}
			\put(95.375,40){\makebox(0,0)[cc]{\scriptsize$y^\prime$}}
			\put(95.25,45.5){\makebox(0,0)[cc]{\scriptsize$x^\prime$}}
			\put(91.375,35.875){\makebox(0,0)[cc]{\scriptsize$y^\ast$}}
			%\emline(84.875,40.625)(89.25,36.875)
			\multiput(84.875,40.625)(.0390625,-.033482143){112}{\line(1,0){.0390625}}
			%\end
			\put(83.125,45.25){\makebox(0,0)[cc]{\scriptsize$x$}}
			\put(83.125,40.25){\makebox(0,0)[cc]{\scriptsize$y$}}
			%\emline(110.996,49.578)(106.754,45.512)
			\multiput(110.996,49.578)(-.035057851,-.033603306){121}{\line(-1,0){.035057851}}
			%\end
			\put(110.996,49.49){\line(1,-1){3.889}}
			\put(114.974,45.512){\line(0,-1){4.95}}
			\put(114.885,40.562){\line(-1,-1){3.977}}
			%\emline(106.754,40.651)(111.085,36.585)
			\multiput(106.754,40.651)(.035793388,-.033603306){121}{\line(1,0){.035793388}}
			%\end
			\put(111.35,49.401){\line(-1,0){.088}}
			\put(111.022,49.307){\circle*{1.226}}
			\put(115.035,40.239){\circle*{1.226}}
			\put(115.184,45.591){\circle*{1.226}}
			\put(107.008,45.442){\circle*{1.226}}
			\put(111.17,36.672){\circle*{1.226}}
			\put(106.875,45.5){\line(1,0){.25}}
			\put(106.988,40.238){\circle*{1.226}}
			\put(106.875,45.5){\line(0,-1){5.375}}
			%\emline(106.875,45.375)(115.125,40.25)
			\multiput(106.875,45.375)(.054276316,-.033717105){152}{\line(1,0){.054276316}}
			%\end
			%\emline(111,36.625)(115.125,45.875)
			\multiput(111,36.625)(.033536585,.075203252){123}{\line(0,1){.075203252}}
			%\end
			\put(112.125,29.5){\makebox(0,0)[cc]{\footnotesize$\Gamma_5$}}
			\put(111.125,51){\makebox(0,0)[cc]{\scriptsize$v^\ast$}}
			\put(105.25,45.375){\makebox(0,0)[cc]{\scriptsize$x$}}
			\put(105.25,40.375){\makebox(0,0)[cc]{\scriptsize$y$}}
			\put(117.125,45.75){\makebox(0,0)[cc]{\scriptsize$x^\prime$}}
			\put(117.25,40){\makebox(0,0)[cc]{\scriptsize$y^\prime$}}
			\put(114,35.5){\makebox(0,0)[cc]{\scriptsize$y^\ast$}}
			\put(111.75,32.625){\makebox(0,0)[cc]{\scriptsize$\lambda_3(\Gamma_5)=0.6180$}}
			%\emline(25.958,20.43)(21.716,16.364)
			\multiput(25.958,20.43)(-.035057851,-.033603306){121}{\line(-1,0){.035057851}}
			%\end
			\put(25.958,20.342){\line(1,-1){3.889}}
			\put(21.716,16.453){\line(0,-1){4.861}}
			\put(29.936,16.364){\line(0,-1){4.95}}
			\put(26.049,20.178){\circle*{1.226}}
			\put(30.062,11.11){\circle*{1.226}}
			\put(30.062,16.164){\circle*{1.226}}
			\put(21.738,16.016){\circle*{1.226}}
			\put(21.887,10.962){\circle*{1.226}}
			\put(26.049,7.245){\circle*{1.226}}
			\put(21.594,16){\line(1,0){8.375}}
			%\emline(21.719,16)(26.094,7)
			\multiput(21.719,16)(.033653846,-.069230769){130}{\line(0,-1){.069230769}}
			%\end
			%\emline(25.844,7.375)(30.094,16.25)
			\multiput(25.844,7.375)(.033730159,.070436508){126}{\line(0,1){.070436508}}
			%\end
			%\emline(25.969,7.125)(30.094,11)
			\multiput(25.969,7.125)(.035869565,.033695652){115}{\line(1,0){.035869565}}
			%\end
			%\emline(21.719,10.875)(26.094,7.25)
			\multiput(21.719,10.875)(.040509259,-.033564815){108}{\line(1,0){.040509259}}
			%\end
			\put(25.5,21.75){\makebox(0,0)[cc]{\scriptsize$v^\ast$}}
			\put(20,15.625){\makebox(0,0)[cc]{\scriptsize$x$}}
			\put(20,10.625){\makebox(0,0)[cc]{\scriptsize$y$}}
			\put(31.75,16.625){\makebox(0,0)[cc]{\scriptsize$x^\prime$}}
			\put(31.875,10.875){\makebox(0,0)[cc]{\scriptsize$y^\prime$}}
			\put(28.25,5.875){\makebox(0,0)[cc]{\scriptsize$y^\ast$}}
			\put(26.375,3.125){\makebox(0,0)[cc]{\scriptsize$\lambda_3(\Gamma_7)=0.6180$}}
			\put(49.25,12.5){\makebox(0,0)[cc]{}}
			\put(46.601,20.216){\line(1,-1){3.889}}
			\put(50.579,16.238){\line(0,-1){4.95}}
			%\emline(42.359,11.377)(46.69,7.311)
			\multiput(42.359,11.377)(.035793388,-.033603306){121}{\line(1,0){.035793388}}
			%\end
			%\emline(42.447,16.327)(50.667,11.2)
			\multiput(42.447,16.327)(.054078947,-.033730263){152}{\line(1,0){.054078947}}
			%\end
			\put(46.955,20.127){\line(-1,0){.088}}
			\put(42.977,16.68){\line(1,0){.088}}
			%\emline(42.889,16.327)(42.182,16.061)
			\multiput(42.889,16.327)(-.088375,-.03325){8}{\line(-1,0){.088375}}
			%\end
			\put(46.478,19.945){\circle*{1.226}}
			\put(50.64,15.783){\circle*{1.226}}
			\put(50.64,10.877){\circle*{1.226}}
			\put(42.613,15.783){\circle*{1.226}}
			\put(42.762,11.026){\circle*{1.226}}
			\put(46.627,7.458){\circle*{1.226}}
			%\emline(42.675,15.875)(46.55,20.25)
			\multiput(42.675,15.875)(.033695652,.038043478){115}{\line(0,1){.038043478}}
			%\end
			%\emline(46.55,7.5)(50.8,15.875)
			\multiput(46.55,7.5)(.033730159,.066468254){126}{\line(0,1){.066468254}}
			%\end
			%\emline(42.675,15.875)(46.675,7.5)
			\multiput(42.675,15.875)(.033613445,-.070378151){119}{\line(0,-1){.070378151}}
			%\end
			\put(42.55,15.875){\line(0,-1){4.625}}
			\put(46.625,21.5){\makebox(0,0)[cc]{\scriptsize$v^\ast$}}
			\put(52.125,15.625){\makebox(0,0)[cc]{\scriptsize$x^\prime$}}
			\put(52.25,9.875){\makebox(0,0)[cc]{\scriptsize$y^\prime$}}
			%\emline(46.5,7.5)(50.625,10.875)
			\multiput(46.5,7.5)(.040841584,.033415842){101}{\line(1,0){.040841584}}
			%\end
			\put(40.75,16){\makebox(0,0)[cc]{\scriptsize$x$}}
			\put(40.625,10.375){\makebox(0,0)[cc]{\scriptsize$y$}}
			\put(48.5,6.125){\makebox(0,0)[cc]{\scriptsize$y^\ast$}}
			\put(46.75,3.25){\makebox(0,0)[cc]{\scriptsize$\lambda_3(\Gamma_8)=0.2798$}}
			%\emline(62.567,16.349)(70.787,11.222)
			\multiput(62.567,16.349)(.054078947,-.033730263){152}{\line(1,0){.054078947}}
			%\end
			\put(62.499,16.138){\line(1,0){8.375}}
			\put(67.346,20.238){\line(1,-1){3.889}}
			\put(71.324,16.26){\line(0,-1){4.95}}
			\put(71.235,11.31){\line(-1,-1){3.977}}
			%\emline(63.104,11.399)(67.435,7.333)
			\multiput(63.104,11.399)(.035793388,-.033603306){121}{\line(1,0){.035793388}}
			%\end
			%\emline(63.28,16.172)(67.346,7.245)
			\multiput(63.28,16.172)(.033603306,-.07377686){121}{\line(0,-1){.07377686}}
			%\end
			\put(67.7,20.149){\line(-1,0){.088}}
			\put(63.722,16.702){\line(1,0){.088}}
			%\emline(63.811,16.702)(63.634,16.349)
			\multiput(63.811,16.702)(-.0295,-.058833){6}{\line(0,-1){.058833}}
			%\end
			%\emline(63.634,16.349)(62.927,16.083)
			\multiput(63.634,16.349)(-.088375,-.03325){8}{\line(-1,0){.088375}}
			%\end
			\put(67.525,20.042){\circle*{1.226}}
			\put(71.539,16.028){\circle*{1.226}}
			\put(71.539,11.272){\circle*{1.226}}
			\put(67.377,7.407){\circle*{1.226}}
			\put(63.363,16.028){\circle*{1.226}}
			\put(63.363,11.272){\circle*{1.226}}
			\put(63.374,16.388){\line(0,-1){.25}}
			\put(63.25,16.25){\line(0,-1){5.125}}
			\put(67.625,21.45){\makebox(0,0)[cc]{\scriptsize$v^\ast$}}
			%\emline(67.25,7.25)(71.5,16.125)
			\multiput(67.25,7.25)(.033730159,.070436508){126}{\line(0,1){.070436508}}
			%\end
			\put(61.625,16.375){\makebox(0,0)[cc]{\scriptsize$x$}}
			\put(61.5,10.75){\makebox(0,0)[cc]{\scriptsize$y$}}
			%\emline(63.125,16.125)(67.625,20.25)
			\multiput(63.125,16.125)(.036585366,.033536585){123}{\line(1,0){.036585366}}
			%\end
			\put(73,16.125){\makebox(0,0)[cc]{\scriptsize$x^\prime$}}
			\put(73,11.125){\makebox(0,0)[cc]{\scriptsize$y^\prime$}}
			\put(69.625,6.125){\makebox(0,0)[cc]{\scriptsize$y^\ast$}}
			\put(67.5,3.5){\makebox(0,0)[cc]{\scriptsize$\lambda_3(\Gamma_9)=0.1589$}}
			\put(26,.125){\makebox(0,0)[cc]{\footnotesize$\Gamma_7$}}
			\put(46.375,0){\makebox(0,0)[cc]{\footnotesize$\Gamma_8$}}
			\put(67.625,0){\makebox(0,0)[cc]{\footnotesize$\Gamma_9$}}
			%\emline(63.25,45.125)(71.375,40.375)
			\multiput(63.25,45.125)(.057624113,-.033687943){141}{\line(1,0){.057624113}}
			%\end
		\end{picture}
	\end{center}
	\caption{Graphs $C_6$, $\Gamma_1$--$\Gamma_{9}$ (for the proof of Lemma~\ref{lem-4-3}).}\label{fig-3}
\end{figure}

\begin{lem}\label{lem-4-3}
Suppose that $G[Y]\cong K_{n_1, n_2, \ldots, n_l}$, where for at least one $i\in \{1,2,\ldots,l\}$ we have $n_i\geq 2$. Let $y, y'$ be two non-adjacent vertices in $G[Y]$. Then $G\in \mathcal{T}_{n}^s(I)$ if $N_X(y)=N_X(y')$ and $G\in \mathcal{T}_{n}^s(II)$, otherwise.
\end{lem}
\begin{proof}
Since $y$ and $y'$ are non-adjacent  and $G[Y]\cong K_{n_1, n_2, \ldots, n_l}$, we have $N_Y(y)=N_Y(y')$. This means  that $N_G(y)=N_G(y')$ if and only if $N_X(y)=N_X(y')$. We distinguish the following two cases.

{\flushleft\it Case 1:} $N_X(y)=N_X(y')$. Here we have $N_G(y)=N_G(y')$, and  so $y$ is a congruent vertex (with respect to $y'$) of I-type. By Lemma \ref{lem-2-4}, we have $p(G)=p(G-y)$ and $\eta(G)=\eta(G-y)+1$. Thus, $G-y\in \mathcal{T}^{s-1}_{n-1}$, and so $G\in \mathcal{T}_{n}^{s}(I)$.

{\flushleft\it Case 2:} $N_X(y)\not=N_X(y')$. Here, one of $N_X(y')\setminus N_X(y)$ and $N_X(y)\setminus N_X(y')$ is non-empty. Without loss of generality, assume that  $x'\in N_X(y')\setminus N_X(y)$, i.e.~$x'\sim y'$ and $x'\not\sim y$. Since $l\geq 2$, there exists $y^*\in Y$ such that $y\sim y^*\sim y'$. If $N_X(y)\neq \emptyset$, then we may take $x\in N_X(y)$ such that the vertices $v^*,x,y,y^\ast, y',x'$ form the cycle $C_6$ in $G$. Up to isomorphism, $G[C_6]$ is one of the 10 induced subgraphs illustrated in Figure~\ref{fig-3}, i.e.
$$G[C_6]\cong\left\{
                                           \begin{array}{ll}
                                             C_6  & ~\hbox{(no additional edges);}\\
                                             \Gamma_1\  \hbox{or}\  \Gamma_2  & ~\hbox{(1 additional edge);} \\
                                             \Gamma_3, \Gamma_4\  \hbox{or}\  \Gamma_5 &~ \hbox{(2 additional edges);} \\
                                             \Gamma_6,\Gamma_7\  \hbox{or}\  \Gamma_8 & ~\hbox{(3 additional edges);} \\
                                             \Gamma_9 & ~\hbox{(4 additional edges).}
                                           \end{array}
                                         \right.
$$
However, $C_6$ and $\Gamma_1$--$\Gamma_9$ are forbidden subgraphs for $G$ by the eigenvalue interlacing. Hence $N_X(y)=\emptyset$, and so $yy^\ast$ is a pendant edge of  $G[X\cup\{y,y^\ast, y', v^*\}]$. By Lemmas~\ref{lem-2-1} and~\ref{lem-2-2}, we have
$$2=p(G)\geq p(G[X\cup \{y, y^\ast,y', v^*\}])=p(G[X\cup \{y', v^*\}])+1. $$
Thus $p(G[X\cup \{y', v^*\}])\leq 1$. Since $G[X\cup \{y', v^*\}]$ is connected, by taking into account  Theorem~\ref{thm-2-2} we conclude that $G[X\cup \{y',
v^*\}]$ is complete multipartite. Consequently, we have $N_X(y')=N_G(v^\ast)=X$ because $v^\ast\nsim y'$. Since $N_Y(y')=N_Y(y)=N_G(y)$, we get that $N_G(y')=N_X(y')\cup N_Y(y')=N_G(v^*)\cup N_G(y)$ is a disjoint union which means that $\{y', v^*, y\}$ is an independent set in $G$, which further means that $y'$ is a congruent vertex (with respect to $v^*$ and $y$) of II-type. Thus $p(G)=p(G-y')$ and $\eta(G)=\eta(G-y')+1$ by Lemma \ref{lem-2-4}, which yields $G-y'\in \mathcal{T}^{s-1}_{n-1}$, and so $G\in \mathcal{T}_{n}^{s}(II)$.
\end{proof}

It remains to consider the class $\mathcal{C}_3$. We partition it into  subclasses determined as follows: $G\in \mathcal{C}_3$ is called $X$-\emph{complete} if $G[X]$ is complete, while for otherwise $G$ is said to be $X$-\emph{incomplete}. We first consider the latter subclass.

\begin{lem}\label{lem-4-4}
Let $G$ be an $X$-incomplete graph of $\mathcal{C}_3$. Then $G\in \mathcal{T}_n^s(I)$ if there exist two non-adjacent vertices $x_1, x_2$ in $G[X]$ such that $N_Y(x_1)=N_Y(x_2)$ and $G\in \mathcal{T}_n^s(III)$, otherwise.
\end{lem}
\begin{proof}
We know that $G[X]$ is not complete and $G[Y]$ is complete with $|Y|=n-t-1\geq 2$. Let $x, x'$ be non-adjacent vertices in $G[X]$. Since $d_G(x),d_G(x')\ge d_G(v^*)$, we have $|N_Y(x)|\ge1$ and $|N_Y(x')|\ge1$. We prove the following two claims.

\smallskip

\noindent{\textbf{Claim~1.}} \textit{If $x\not\sim x'$ in $G[X]$, then one of $N_{Y}(x), N_{Y}(x')$ includes the other. Moreover, if $N_{Y}(x)\subset N_{Y}(x')$, then $|N_Y(x)|=1$ and $N_Y(x')=Y$.}

\smallskip

\noindent{\textit{Proof of Claim~1}.}
Assume by way of contradiction that $y\in N_Y(x)\setminus N_Y(x')$ and $y'\in N_Y(x')\setminus N_Y(x)$. Then $G[v^*, x, y, y',x']$ $\cong C_5$, which
 leads to $\lambda_3(G)\ge \lambda_3(C_5)=0.6180$, a contradiction.

 Assume now that $N_{Y}(x)\subset N_{Y}(x')$. There exists $y^\ast\in N_{Y}(x')\setminus N_{Y}(x)$. If $|N_Y(x)|\geq 2$, say $\{y, y'\}\subseteq N_Y(x)$, then $x'\sim y, y'$. Thus, $G[v^*,x,y,y',y^\ast,x']\cong \Gamma_{10}$ (see Figure~\ref{fig-4}), but $\lambda_3(\Gamma_{10})=0.1505$, a contradiction. Hence, $|N_Y(x)|=1$ and we may assume that $N_Y(x)=\{y\}$. If $N_Y(x')\neq Y$, then there exists $y'\in Y\setminus N_Y(x')$. Recall that $y^\ast\in N_{Y}(x')\setminus N_{Y}(x)$, and so we have $G[v^*,x,y,y',y^\ast,x']\cong \Gamma_{11}$ (see Figure~\ref{fig-4}), but $\lambda_3(\Gamma_{11})=0.1830$, a contradiction. Thus $N_Y(x')=Y$. This proves the claim.

\smallskip

\noindent{\textbf{Claim~2.}}\textit{
Let $x\not\sim x'$ in $G[X]$. The following holds true.
\begin{enumerate}[{\rm (a)}]
\item If $N_{Y}(x)=\{y\}$, then $N_{X}(x)=X\setminus\{x, x'\}$;
\item If $\{y\}=N_{Y}(x)\subset N_{Y}(x')=Y$, then $N_X(x')=N_X(y)\setminus \{x, x'\}$.
\end{enumerate}}
\smallskip

\noindent{\textit{Proof of Claim~2}.}
Since $d_G(x)\ge |X|=t$, $x\not\sim x'$ and $N_Y(x)=\{y\}$, we see that $x$ is also a vertex with minimum degree. Hence, $N_X(x)=X\setminus\{x,x'\}$ and (a) follows.

Assume that $x^*\in N_X(x')$. We have $x^*\sim x'$ and, from (a),  $x\sim x^*$. Since $|Y|\ge 2$,  apart from $y$, there exists $ y'\in N_Y(x')$. Recall that $x$ has only one neighbour $y$ in $Y$ and $y\sim y'$. Now, if $x^*\notin N_X(y)$, then
$$G[v^*,x,y,y',x',x^*]\cong \left\{\begin{array}{ll}\Gamma_{12}& \mbox{whenever $x^*\nsim y'$ (see Figure \ref{fig-4}); }\\
 \Gamma_{13} & \mbox{whenever $x^*\sim y'$ (see Figure \ref{fig-4}).}
\end{array}\right.
$$
But $\Gamma_{12}$ and $\Gamma_{13}$ are forbidden subgraphs for $G$, a contradiction. Thus $x^*\in N_X(y)$, which yields $x^*\in N_X(y)\setminus \{x, x'\}$. In other words, $N_X(x')\subseteq N_X(y)\setminus \{x, x'\}$. To prove that these sets are equal, we assume that $x^{**}\in N_X(y)\setminus \{x, x'\}$. From (a) we have $x^{**}\sim y,x,v^*$. Now, if $x^{**}\notin N_X(x')$, then  $$G[v^*,x,y,y',x',x^{**}]\cong \left\{\begin{array}{ll}\Gamma_{14}& \mbox{ whenever $x^{**}\nsim y'$ (see Figure \ref{fig-4});}\\
 \Gamma_{15}& \mbox{ whenever $x^{**}\sim y'$ (see Figure \ref{fig-4}).}
\end{array}\right.
$$

But $\Gamma_{14}$ and $\Gamma_{15}$ are forbidden for $G$, a contradiction. Hence $x^{**}\in N_X(x')$, and so $N_X(x')=N_X(y)\setminus \{x,x'\}$. This completes the proof of Claim 2.

\smallskip

We now distinguish the following cases.

{\flushleft\it Case 1:} There exists a pair of non-adjacent vertices $x_1,x_2\in X$ such that $N_Y(x_1)=N_Y(x_2)$.

We first show that in this case there must be $N_X(x_1)=N_X(x_2)$. If $N_Y(x_1)=N_Y(x_2)=\{y\}$ for $y\in Y$, then $N_X(x_1)=N_X(x_2)=X\setminus\{x_1, x_2\}$ by Claim 2(a), as desired.

Now suppose that $\{y, y'\}\subseteq N_Y(x_1)=N_Y(x_2)$. By way of contradiction, assume that $x^\ast\in N_X(x_2)\setminus N_X(x_1)$, i.e.~$x^\ast\sim x_2$ and $x^\ast\not\sim x_1$. By Claim 1, we have $N_Y(x^\ast)\subseteq N_Y(x_1)$ (since for otherwise we would have $N_Y(x_1)\subset N_Y(x^\ast)$, and then $|N_Y(x_1)|=1$, a contradiction). If $N_Y(x^\ast)$ is a proper subset of $N_Y(x_1)$, then  Claim~1 yields  $N_Y(x_1)=Y$ and $|N_Y(x^\ast)|=1$. Accordingly, let $N_Y(x^\ast)=\{y^*\}$. Then we have $N_X(x_1)=N_X(y^*)\setminus \{x^*, x_1\}$ by Claim 2(b). On the other hand, from $y^*\in N_Y(x^\ast) \subset N_Y(x_1)=N_Y(x_2)$, we have $y^*\sim x_2$, and so $x_2\in N_X(y^*)\setminus \{x^*, x_1\}=N_X(x_1)$ which gives $x_2\sim x_1$, a contradiction. If $N_Y(x^\ast)=N_Y(x_1)$, then $x^\ast\sim y, y'$, and so $G[v^*,x_1,y,y',x^\ast,x_2]\cong \Gamma_{16}$ (see Figure~\ref{fig-4}), which is a contradiction as before.

Hence, $N_{X}(x_1)= N_{X}(x_2)$ and thus  $N_G(x_1)=N_G(x_2)$ by the assumption stated in this case. Thus $x_1$ is a congruent vertex (with respect to $x_2$) of I-type. By Lemma~\ref{lem-2-4}, $p(G)=p(G-x_1)$ and $\eta(G)=\eta(G-x_1)+1$. Thus $G-x_1\in \mathcal{T}^{s-1}_{n-1}$, which  implies that $G\in \mathcal{T}_n^s(I)$.

{\flushleft\it Case 2:}
For every pair of non-adjacent vertices $x,x'$ in $G[X]$, we have $N_Y(x)\not=N_Y(x')$.

By Claim 1, we may assume that $N_Y(x)\subset N_Y(x')=Y$, and then we have $|N_Y(x)|=1$, say $N_Y(x)=\{y\}$. Thus $y\sim x, x'$ and $xv^\ast x'y$ induce the quadrangle $C_4$. We need to verify that $C_4$ is congruent. By Claim~2(a), we have $N_G(x)\setminus \{v^\ast, y\}=X\setminus\{x', x\}=N_G(v^\ast)\setminus \{x', x\}$. By Claim~2(b), we have $N_G(x')\setminus\{y, v^\ast\}=N_X(x')\cup (Y\setminus \{y\})=N_G(y)\setminus \{x,x'\}$. Altogether, the quadrangle $C_4=xv^\ast x'y$ is congruent, where $xv^\ast$ and $x'y$ act as congruent edges. This implies that $x, v^\ast, x', y$ are congruent vertices of III-type. By Lemma \ref{lem-2-4}, we have $p(G)=p(G-x)$ and $\eta(G)=\eta(G-x)+1$. Therefore $G-x\in \mathcal{T}^{s-1}_{n-1}$, which implies that $G\in \mathcal{T}_n^s(III)$. The proof is complete.\end{proof}

\begin{figure}[t]
\begin{center}
\unitlength 1.2mm % = 2.845pt
\linethickness{0.5pt}
\ifx\plotpoint\undefined\newsavebox{\plotpoint}\fi % GNUPLOT compatibility
\setlength{\unitlength}{3.5pt}
\begin{picture}(105,50.5)(0,0)
\put(14.25,44.375){\makebox(0,0)[cc]{}}
%\emline(6.794,49.461)(2.552,45.395)
\multiput(6.794,49.461)(-.035057851,-.033603306){121}{\line(-1,0){.035057851}}
%\end
%\emline(2.552,40.534)(6.883,36.468)
\multiput(2.552,40.534)(.035793388,-.033603306){121}{\line(1,0){.035793388}}
%\end
\put(2.728,45.396){\line(0,-1){4.685}}
\put(10.772,45.485){\line(0,-1){4.861}}
\put(6.752,49.086){\circle*{1.226}}
\put(10.616,45.072){\circle*{1.226}}
\put(10.616,40.167){\circle*{1.226}}
\put(6.9,36.451){\circle*{1.226}}
\put(2.738,40.167){\circle*{1.226}}
\put(2.738,45.072){\circle*{1.226}}
%\emline(6.637,49.141)(10.637,45.266)
\multiput(6.637,49.141)(.034782609,-.033695652){115}{\line(1,0){.034782609}}
%\end
\put(8.25,28){\makebox(0,0)[cc]{\footnotesize$\Gamma_{10}$}}
\put(3.875,31.625){\makebox(0,0)[cc]{\tiny$\lambda_3(\Gamma_{10})=0.1505$}}
\put(6.25,50.375){\makebox(0,0)[cc]{\scriptsize$v^\ast$}}
\put(2.625,40.375){\line(1,0){8}}
%\emline(2.5,45.375)(7,36.5)
\multiput(2.5,45.375)(.03358209,-.066231343){134}{\line(0,-1){.066231343}}
%\end
\put(10.625,45.5){\line(-3,-2){8.25}}
%\emline(10.75,45.125)(6.875,36.5)
\multiput(10.75,45.125)(-.033695652,-.075){115}{\line(0,-1){.075}}
%\end
\put(1.125,40.625){\makebox(0,0)[cc]{\scriptsize$y$}}
\put(1.125,45.25){\makebox(0,0)[cc]{\scriptsize$x$}}
\put(12.25,41.5){\makebox(0,0)[cc]{\scriptsize$y^\ast$}}
\put(9,35.75){\makebox(0,0)[cc]{\scriptsize$y^\prime$}}
%\emline(6.875,36.625)(10.75,40.25)
\multiput(6.875,36.625)(.03587963,.033564815){108}{\line(1,0){.03587963}}
%\end
\put(12.375,45.875){\makebox(0,0)[cc]{\scriptsize$x^\prime$}}
\put(23.641,49.313){\line(1,-1){3.889}}
\put(27.619,45.335){\line(0,-1){4.95}}
\put(23.775,49.178){\circle*{1.226}}
\put(27.491,45.164){\circle*{1.226}}
\put(27.64,40.556){\circle*{1.226}}
\put(23.924,36.543){\circle*{1.226}}
\put(19.613,45.164){\circle*{1.226}}
\put(19.613,40.488){\circle*{1.226}}
\put(19.625,45.125){\line(0,-1){4.75}}
%\emline(19.5,40.375)(23.875,36.375)
\multiput(19.5,40.375)(.036764706,-.033613445){119}{\line(1,0){.036764706}}
%\end
\put(23,31.875){\makebox(0,0)[cc]{\tiny$\lambda_3(\Gamma_{11})=0.1830$}}
\put(18.125,45.25){\makebox(0,0)[cc]{\scriptsize$x$}}
\put(18,40.75){\makebox(0,0)[cc]{\scriptsize$y$}}
\put(24.125,28.125){\makebox(0,0)[cc]{\footnotesize$\Gamma_{11}$}}
\put(19.625,40.625){\line(1,0){8.375}}
%\emline(19.375,40.625)(27.5,45.375)
\multiput(19.375,40.625)(.057624113,.033687943){141}{\line(1,0){.057624113}}
%\end
\put(24.125,50.8){\makebox(0,0)[cc]{\scriptsize$v^\ast$}}
\put(29.25,45.875){\makebox(0,0)[cc]{\scriptsize$x^\prime$}}
%\emline(19.625,45.25)(23.75,49.25)
\multiput(19.625,45.25)(.034663866,.033613445){119}{\line(1,0){.034663866}}
%\end
%\emline(23.875,36.5)(27.625,40.5)
\multiput(23.875,36.5)(.033482143,.035714286){112}{\line(0,1){.035714286}}
%\end
\put(29.625,41.5){\makebox(0,0)[cc]{\scriptsize$y^\ast$}}
\put(26.125,35.625){\makebox(0,0)[cc]{\scriptsize$y^\prime$}}
\put(60.125,28.125){\makebox(0,0)[cc]{\footnotesize$\Gamma_{13}$}}
\put(4.25,3.75){\makebox(0,0)[cc]{\tiny$\lambda_3(\Gamma_{16})=0.2679$}}
\put(6.875,0){\makebox(0,0)[cc]{\footnotesize$\Gamma_{16}$}}
\put(42,32.125){\makebox(0,0)[cc]{\tiny$\lambda_3(\Gamma_{12})=0.1096$}}
\put(42,28.125){\makebox(0,0)[cc]{\footnotesize$\Gamma_{12}$}}
\put(40.924,48.611){\circle*{1.226}}
\put(36.762,44.746){\circle*{1.226}}
\put(35.125,44.875){\makebox(0,0)[cc]{\scriptsize$x$}}
\put(41.5,50.425){\makebox(0,0)[cc]{\scriptsize$v^\ast$}}
\put(44.153,36.451){\circle*{1.226}}
\put(46.375,35.25){\makebox(0,0)[cc]{\scriptsize$y^\prime$}}
\put(38.613,36.488){\circle*{1.226}}
\put(37,36.796){\makebox(0,0)[cc]{\scriptsize$y$}}
\put(46.488,47.113){\circle*{1.226}}
\put(48.552,47.867){\makebox(0,0)[cc]{\scriptsize$x^\ast$}}
\put(46.488,42.238){\circle*{1.226}}
\put(48.5,40.625){\makebox(0,0)[cc]{\scriptsize$x^\prime$}}
\put(40.875,48.875){\line(-1,-1){4.25}}
%\emline(36.625,44.875)(38.625,36.375)
\multiput(36.625,44.875)(.03333333,-.14166667){60}{\line(0,-1){.14166667}}
%\end
\put(38.375,36.625){\line(1,0){5.875}}
%\emline(46.5,42.375)(44.25,36.625)
\multiput(46.5,42.375)(-.03358209,-.0858209){67}{\line(0,-1){.0858209}}
%\end
\put(46.375,47.125){\line(0,-1){5.375}}
%\emline(40.625,48.75)(46.375,47.125)
\multiput(40.625,48.75)(.11734694,-.03316327){49}{\line(1,0){.11734694}}
%\end
%\emline(40.75,48.625)(46.375,42.25)
\multiput(40.75,48.625)(.033682635,-.038173653){167}{\line(0,-1){.038173653}}
%\end
%\emline(36.625,44.75)(46.625,47)
\multiput(36.625,44.75)(.14925373,.03358209){67}{\line(1,0){.14925373}}
%\end
%\emline(38.5,36.5)(46.5,42.25)
\multiput(38.5,36.5)(.046783626,.033625731){171}{\line(1,0){.046783626}}
%\end
\put(59.424,48.486){\circle*{1.226}}
\put(55.262,44.621){\circle*{1.226}}
\put(53.625,44.75){\makebox(0,0)[cc]{\scriptsize$x$}}
\put(60,50.2){\makebox(0,0)[cc]{\scriptsize$v^\ast$}}
\put(62.653,36.326){\circle*{1.226}}
\put(64.875,35.125){\makebox(0,0)[cc]{\scriptsize$y^\prime$}}
\put(57.113,36.363){\circle*{1.226}}
\put(55.5,36.671){\makebox(0,0)[cc]{\scriptsize$y$}}
\put(64.988,46.988){\circle*{1.226}}
\put(66.952,47.742){\makebox(0,0)[cc]{\scriptsize$x^\ast$}}
\put(64.988,42.113){\circle*{1.226}}
\put(67,40.5){\makebox(0,0)[cc]{\scriptsize$x^\prime$}}
\put(59.375,48.75){\line(-1,-1){4.25}}
%\emline(55.125,44.75)(57.125,36.25)
\multiput(55.125,44.75)(.03333333,-.14166667){60}{\line(0,-1){.14166667}}
%\end
\put(56.875,36.5){\line(1,0){5.875}}
%\emline(65,42.25)(62.75,36.5)
\multiput(65,42.25)(-.03358209,-.0858209){67}{\line(0,-1){.0858209}}
%\end
\put(64.875,47){\line(0,-1){5.375}}
%\emline(59.125,48.625)(64.875,47)
\multiput(59.125,48.625)(.11734694,-.03316327){49}{\line(1,0){.11734694}}
%\end
%\emline(59.25,48.5)(64.875,42.125)
\multiput(59.25,48.5)(.033682635,-.038173653){167}{\line(0,-1){.038173653}}
%\end
%\emline(55.125,44.625)(65.125,46.875)
\multiput(55.125,44.625)(.14925373,.03358209){67}{\line(1,0){.14925373}}
%\end
%\emline(57,36.375)(65,42.125)
\multiput(57,36.375)(.046783626,.033625731){171}{\line(1,0){.046783626}}
%\end
%\emline(64.75,47.125)(62.375,36.375)
\multiput(64.75,47.125)(-.0334507,-.15140845){71}{\line(0,-1){.15140845}}
%\end
\put(77.924,48.486){\circle*{1.226}}
\put(73.762,44.621){\circle*{1.226}}
\put(72.125,44.75){\makebox(0,0)[cc]{\scriptsize$x$}}
\put(78.99,50.1){\makebox(0,0)[cc]{\scriptsize$v^\ast$}}
\put(81.153,36.326){\circle*{1.226}}
\put(83.375,35.125){\makebox(0,0)[cc]{\scriptsize$y^\prime$}}
\put(75.613,36.363){\circle*{1.226}}
\put(74,36.671){\makebox(0,0)[cc]{\scriptsize$y$}}
\put(83.488,46.988){\circle*{1.226}}
\put(83.488,42.113){\circle*{1.226}}
\put(85.5,40.5){\makebox(0,0)[cc]{\scriptsize$x^\prime$}}
\put(85.9,48.375){\makebox(0,0)[cc]{\scriptsize$x^{**}$}}
\put(78,48.625){\line(-1,-1){4.25}}
%\emline(77.875,48.5)(83.5,42.125)
\multiput(77.875,48.5)(.033682635,-.038173653){167}{\line(0,-1){.038173653}}
%\end
%\emline(78,48.625)(83.375,47)
\multiput(78,48.625)(.10969388,-.03316327){49}{\line(1,0){.10969388}}
%\end
%\emline(73.75,44.625)(75.625,36.375)
\multiput(73.75,44.625)(.03348214,-.14732143){56}{\line(0,-1){.14732143}}
%\end
\put(75.5,36.375){\line(1,0){5.75}}
%\emline(81.25,36.375)(83.5,42.25)
\multiput(81.25,36.375)(.03358209,.08768657){67}{\line(0,1){.08768657}}
%\end
%\emline(83.5,42.25)(75.5,36.5)
\multiput(83.5,42.25)(-.046783626,-.033625731){171}{\line(-1,0){.046783626}}
%\end
%\emline(83.5,47)(75.5,36.625)
\multiput(83.5,47)(-.033613445,-.043592437){238}{\line(0,-1){.043592437}}
%\end
%\emline(73.75,44.75)(83.5,46.875)
\multiput(73.75,44.75)(.1547619,.03373016){63}{\line(1,0){.1547619}}
%\end
\put(80,32.125){\makebox(0,0)[cc]{\tiny$\lambda_3(\Gamma_{14})=0.1096$}}
\put(61,32){\makebox(0,0)[cc]{\tiny$\lambda_3(\Gamma_{13})=0.1873$}}
\put(78.375,28.25){\makebox(0,0)[cc]{\footnotesize$\Gamma_{14}$}}
\put(97.424,48.611){\circle*{1.226}}
\put(93.262,44.746){\circle*{1.226}}
\put(91.625,44.875){\makebox(0,0)[cc]{\scriptsize$x$}}
\put(98,50.725){\makebox(0,0)[cc]{\scriptsize$v^\ast$}}
\put(100.653,36.451){\circle*{1.226}}
\put(102.875,35.25){\makebox(0,0)[cc]{\scriptsize$y^\prime$}}
\put(95.113,36.488){\circle*{1.226}}
\put(93.5,36.796){\makebox(0,0)[cc]{\scriptsize$y$}}
\put(102.988,47.113){\circle*{1.226}}
\put(102.988,42.238){\circle*{1.226}}
\put(105,40.625){\makebox(0,0)[cc]{\scriptsize$x^\prime$}}
\put(105.1,48.5){\makebox(0,0)[cc]{\scriptsize$x^{**}$}}
\put(97.5,48.75){\line(-1,-1){4.25}}
%\emline(97.375,48.625)(103,42.25)
\multiput(97.375,48.625)(.033682635,-.038173653){167}{\line(0,-1){.038173653}}
%\end
%\emline(97.5,48.75)(102.875,47.125)
\multiput(97.5,48.75)(.10969388,-.03316327){49}{\line(1,0){.10969388}}
%\end
%\emline(93.25,44.75)(95.125,36.5)
\multiput(93.25,44.75)(.03348214,-.14732143){56}{\line(0,-1){.14732143}}
%\end
\put(95,36.5){\line(1,0){5.75}}
%\emline(100.75,36.5)(103,42.375)
\multiput(100.75,36.5)(.03358209,.08768657){67}{\line(0,1){.08768657}}
%\end
%\emline(103,42.375)(95,36.625)
\multiput(103,42.375)(-.046783626,-.033625731){171}{\line(-1,0){.046783626}}
%\end
%\emline(103,47.125)(95,36.75)
\multiput(103,47.125)(-.033613445,-.043592437){238}{\line(0,-1){.043592437}}
%\end
%\emline(93.25,44.875)(103,47)
\multiput(93.25,44.875)(.1547619,.03373016){63}{\line(1,0){.1547619}}
%\end
\put(100,32.125){\makebox(0,0)[cc]{\tiny$\lambda_3(\Gamma_{15})=0.1873$}}
\put(98.625,28.25){\makebox(0,0)[cc]{\footnotesize$\Gamma_{15}$}}
%\emline(102.875,47.25)(100.625,36.375)
\multiput(102.875,47.25)(-.03358209,-.16231343){67}{\line(0,-1){.16231343}}
%\end
\put(5.799,20.361){\circle*{1.226}}
\put(1.637,16.496){\circle*{1.226}}
\put(6.675,22.1){\makebox(0,0)[cc]{\scriptsize$v^\ast$}}
\put(9.028,8.201){\circle*{1.226}}
\put(11.25,7){\makebox(0,0)[cc]{\scriptsize$y^\prime$}}
\put(3.488,8.238){\circle*{1.226}}
\put(1.875,8.546){\makebox(0,0)[cc]{\scriptsize$y$}}
\put(11.363,18.863){\circle*{1.226}}
\put(11.363,13.988){\circle*{1.226}}
\put(0,16.625){\makebox(0,0)[cc]{\scriptsize$x_1$}}
\put(12.975,15.55){\makebox(0,0)[cc]{\scriptsize$x^{*}$}}
\put(13.875,18.625){\makebox(0,0)[cc]{\scriptsize$x_2$}}
%\emline(5.875,20.375)(1.5,16.375)
\multiput(5.875,20.375)(-.036764706,-.033613445){119}{\line(-1,0){.036764706}}
%\end
%\emline(5.625,20.375)(11.375,18.875)
\multiput(5.625,20.375)(.12777778,-.03333333){45}{\line(1,0){.12777778}}
%\end
%\emline(5.375,20.375)(11.25,14)
\multiput(5.375,20.375)(.033571429,-.036428571){175}{\line(0,-1){.036428571}}
%\end
%\emline(1.5,16.5)(3.5,8.125)
\multiput(1.5,16.5)(.03333333,-.13958333){60}{\line(0,-1){.13958333}}
%\end
%\emline(1.375,16.875)(9,8.25)
\multiput(1.375,16.875)(.033738938,-.038163717){226}{\line(0,-1){.038163717}}
%\end
\put(3.25,8.125){\line(1,0){5.875}}
%\emline(3.375,8.125)(11.5,14)
\multiput(3.375,8.125)(.046428571,.033571429){175}{\line(1,0){.046428571}}
%\end
%\emline(11.375,18.875)(3.5,8.25)
\multiput(11.375,18.875)(-.033653846,-.045405983){234}{\line(0,-1){.045405983}}
%\end
%\emline(8.875,8.125)(11.5,14.125)
\multiput(8.875,8.125)(.03365385,.07692308){78}{\line(0,1){.07692308}}
%\end
%\emline(11.25,19)(8.75,8.125)
\multiput(11.25,19)(-.03333333,-.145){75}{\line(0,-1){.145}}
%\end
\put(11.25,19){\line(0,-1){5.25}}
\put(23.563,21.37){\line(1,-1){3.889}}
\put(19.321,17.481){\line(1,0){8.22}}
\put(19.321,17.481){\line(0,-1){4.861}}
\put(27.541,17.392){\line(0,-1){4.95}}
%\emline(19.321,12.531)(23.652,8.465)
\multiput(19.321,12.531)(.035793388,-.033603306){121}{\line(1,0){.035793388}}
%\end
\put(19.32,17.481){\line(1,0){8.22}}
\put(19.409,12.531){\line(1,0){8.132}}
\put(23.549,21.111){\circle*{1.226}}
\put(27.711,17.246){\circle*{1.226}}
\put(27.562,12.192){\circle*{1.226}}
\put(19.238,12.192){\circle*{1.226}}
\put(19.387,17.246){\circle*{1.226}}
\put(23.403,8.701){\circle*{1.226}}
%\emline(23.415,8.838)(27.79,12.088)
\multiput(23.415,8.838)(.045103093,.033505155){97}{\line(1,0){.045103093}}
%\end
\put(24,22.8){\makebox(0,0)[cc]{\scriptsize$v^\ast$}}
\put(17.75,17.375){\makebox(0,0)[cc]{\scriptsize$x$}}
\put(17.625,12.5){\makebox(0,0)[cc]{\scriptsize$y$}}
\put(27.625,17.25){\line(-1,-2){4.375}}
\put(24.875,7.375){\makebox(0,0)[cc]{\scriptsize$y^\ast$}}
\put(29.375,13){\makebox(0,0)[cc]{\scriptsize$y^\prime$}}
\put(29.375,17.75){\makebox(0,0)[cc]{\scriptsize$x^\prime$}}
%\emline(19.125,17.375)(23.5,21.125)
\multiput(19.125,17.375)(.0390625,.033482143){112}{\line(1,0){.0390625}}
%\end
\put(23.25,0){\makebox(0,0)[cc]{\footnotesize$\Gamma_{17}$}}
\put(24.125,3.625){\makebox(0,0)[cc]{\scriptsize$\lambda_3(\Gamma_{17})=0.1096$}}
\put(41.174,20.736){\circle*{1.226}}
\put(37.012,16.871){\circle*{1.226}}
\put(35.375,17){\makebox(0,0)[cc]{\scriptsize$x$}}
\put(41.75,22.55){\makebox(0,0)[cc]{\scriptsize$v^\ast$}}
\put(44.403,8.576){\circle*{1.226}}
\put(46.625,7.375){\makebox(0,0)[cc]{\scriptsize$y^\prime$}}
\put(38.863,8.613){\circle*{1.226}}
\put(37.25,8.921){\makebox(0,0)[cc]{\scriptsize$y$}}
\put(46.738,19.238){\circle*{1.226}}
\put(48.702,19.992){\makebox(0,0)[cc]{\scriptsize$x^\ast$}}
\put(46.738,14.363){\circle*{1.226}}
\put(48.75,12.75){\makebox(0,0)[cc]{\scriptsize$x^\prime$}}
\put(41.125,20.75){\line(-1,-1){4.25}}
%\emline(37,17)(38.875,8.5)
\multiput(37,17)(.03348214,-.15178571){56}{\line(0,-1){.15178571}}
%\end
\put(38.875,8.5){\line(1,0){5.5}}
%\emline(44.25,8.625)(46.875,14.5)
\multiput(44.25,8.625)(.03365385,.07532051){78}{\line(0,1){.07532051}}
%\end
\put(46.75,14.5){\line(0,1){5}}
%\emline(41,20.875)(46.75,19.25)
\multiput(41,20.875)(.11734694,-.03316327){49}{\line(1,0){.11734694}}
%\end
%\emline(41,20.875)(46.75,14.25)
\multiput(41,20.875)(.033625731,-.03874269){171}{\line(0,-1){.03874269}}
%\end
%\emline(36.875,16.875)(46.625,19.125)
\multiput(36.875,16.875)(.14552239,.03358209){67}{\line(1,0){.14552239}}
%\end
%\emline(36.75,16.875)(46.75,14.5)
\multiput(36.75,16.875)(.14084507,-.0334507){71}{\line(1,0){.14084507}}
%\end
%\emline(46.75,19.25)(38.875,8.5)
\multiput(46.75,19.25)(-.033653846,-.045940171){234}{\line(0,-1){.045940171}}
%\end
\put(44.5,3.625){\makebox(0,0)[cc]{\scriptsize$\lambda_3(\Gamma_{18})=0.1505$}}
\put(42.25,.25){\makebox(0,0)[cc]{\footnotesize$\Gamma_{18}$}}
\end{picture}
\end{center}
\caption{The graphs $\Gamma_{10}$--$\Gamma_{18}$ (for the proofs of Lemmas~\ref{lem-4-4} and~\ref{lem-4-5}).}\label{fig-4}
\end{figure}

It remains to consider $X$-complete graphs of $\mathcal{C}_3$. Such a graph is called \emph{reduced} if, for every $x,x'\in X$, we have $N_Y(x)\subseteq N_Y(x')$ or $N_Y(x')\subseteq N_Y(x)$, and \emph{unreduced} otherwise. For unreduced graphs, there exist vertices $x,x'\in X$ such that $N_Y(x)\setminus N_Y(x')\not=\emptyset$ and $N_Y(x')\setminus N_Y(x)\not=\emptyset$. In the same spirit, such vertices are called \emph{unreduced vertices}. For a reduced graph we may assume that $\emptyset=N_Y(v^\ast)\subseteq N_Y(x_1)\subseteq N_Y(x_2)\subseteq\cdots\subseteq N_Y(x_t)$ for $X=\{x_1, x_2,\ldots, x_t\}$.

\begin{lem}\label{lem-4-5}
Every unreduced graph of $\mathcal{C}_3$ belongs to $\mathcal{T}_n^s(III)$.
\end{lem}
\begin{proof}
Let $x,x'\in X$ be unreduced vertices. There exist $y\in N_Y(x)\setminus N_Y(x')$ and $y'\in N_Y(x')\setminus N_Y(x)$, and so $G[x,x',y',y]\cong C_4$. It suffices to verify that $C_4$ is congruent. If there exists $y^*\in N_Y(x') \setminus N_Y(x)$ other than $y'$, then $G[v^\ast,x,y,y^*,y',x'] \cong \Gamma_{17}$ (see Figure \ref{fig-4}) is  forbidden for $G$. Hence $N_Y(x')\setminus N_Y(x)=\{y'\}$. Similarly, we have $N_Y(x)\setminus N_Y(x')=\{y\}$. Thus $N_Y(x)\setminus \{y\}=N_Y(x')\setminus\{y'\}$, and then
$$N_G(x)\setminus\{y, x'\}=(X\setminus \{x, x'\})\cup (N_Y(x)\setminus\{y\})=N_G(x')\setminus \{x, y'\}.$$

On the other hand, $x\in N_X(y)\setminus N_X(y')$ and $x'\in N_X(y')\setminus N_X(y)$. If there exists $x^\ast\in N_X(y)\setminus N_X(y')$ other than $x$, then $G[v^\ast,x,y,y',x',x^\ast]\cong \Gamma_{18}$ (see Figure \ref{fig-4}) is  forbidden for $G$. Hence $N_X(y)\setminus N_X(y')=\{x\}$. Similarly, we have $N_X(y')\setminus N_X(y)=\{x'\}$. Thus $N_X(y)\setminus \{x\}=N_X(y')\setminus\{x'\}$, and then
$$N_G(y)\setminus\{y', x\}=(Y\setminus \{y, y'\})\cup (N_X(y)\setminus\{x\})=N_G(y')\setminus \{x', y\}.$$

Hence, the quadrangle $xx'y'y$ is congruent, where $xx'$ and $y'y$ is a pair of congruent edges. It follows that $x$ is a congruent vertex of III-type, which leads to the desired result.
\end{proof}

In what follows we will see that reduced graphs are not contained in $\mathcal{T}^s_n(I)\cup\mathcal{T}^s_n(II)\cup\mathcal{T}_n^s(III)$, i.e.~they are underivable. In order to characterize them, we need the notion of a canonical graph introduced in \cite{Per}. For a graph $G$, we say that $u,v\in V(G)$ are in relation $\rho$, designated by $u\rho v$, if and only if $u\sim v$ and $N_G(u)\setminus v=N_G(v)\setminus u$. Clearly, $\rho$ is symmetric and transitive. Accordingly, the vertex set $V(G)$ is partitioned as $V(G)=V_1\cup V_2\cup\cdots\cup V_k$, where $v_i\in V_i$ and $V_i=\{x\in V(G)\mid x\rho v_i\}$. By  definition of $\rho$, we get that $V_i$ induces a clique $K_{n_i}$ ($n_i=|V_i|$) for each $i$, and that every vertex of $V_i$ is adjacent to every vertex of $V_j$ if and only if $v_i\sim v_j$ in $G$. The \emph{canonical graph} of $G$, denoted by $G^c$, is the induced subgraph $G[\{v_1,v_2,\ldots,v_k\}]$. Then we have
\begin{equation}\label{de-eq-1}G\cong G^c[K_{n_1}, K_{n_2},\ldots,K_{n_k}]\end{equation}
which is just the generalized lexicographic product of $G^c$ with $K_{n_1}$, $K_{n_2}$,\ldots,$K_{n_k}$.

In the forthcoming Lemma~\ref{lem-4-6} we prove that the canonical graph $G^c$ of a reduced $X$-complete graph $G$ is just the reduced half-complete graph $G_k$ (defined in Section~\ref{sec:prel}) for some positive integer~$k$. The following lemma from~\cite{M.R.Oboudi2} is needed.

\begin{lem}\cite{M.R.Oboudi2}\label{lem-4-7}
Let $G\cong P_3[K_{n_1}, K_{n_2}, K_{n_3}]$ and $P_3\cong G_3$ be a reduced half-complete graph, where $n_1, n_2, n_3$ are  positive integers. The following holds true.
\begin{enumerate}[{\rm (i)}]
\item If $n_1=n_2=n_3=1$, then $G\cong P_3$ and $\lambda_3(G)=-\sqrt{2}$;
\item If $n_1=n_2=1$ and $n_3\geq 2$, then $\lambda_3(G)=-1$;
\item If $n_1n_2> 1$, then $\lambda_3(G)=-1$.
\end{enumerate}
\end{lem}

We proceed with the announced lemma.

\begin{lem}\label{lem-4-6}
Let $G$ be a reduced $X$-complete graph. Then there exists a reduced half-complete graph~$G_k$  such that $G\cong G_k[K_{n_1}, K_{n_2},\ldots,K_{n_k}]$, where $4\le k\le 14$ and $n_1+n_2+\cdots+n_k=n$.
\end{lem}
\begin{proof}
From Eq.~(\ref{de-eq-1}) we have $G\cong G^c[K_{n_1}, K_{n_2},\ldots,$ $K_{n_k}]$, where $G^c\cong G[\{v_1, v_2, \ldots,$ $v_k\}]$ and $V_i=\{x\in V(G)\mid x\rho v_i\}$ induces the clique $K_{n_i}$. We need to show that the canonical graph  $G^c$ is just a reduced half-complete graph $G_k$ for some $k$. Without loss of generality, we assume that $v_1=v^*$ (a vertex with minimum degree of $G^c$) and denote $t_c=d_{G^c}(v_1)$. Let $X_c=N_{G^c}(v_1)=\{x_1,x_2,\ldots,x_{t_c}\}\subset \{v_1, v_2, \ldots,v_k\}$ and  $Y_c=\{v_2,v_3,\ldots,v_k\}\backslash X_c=\{y_1,y_2, \ldots,y_{k-t_c-1}\}$. We see that $X_c$ is a subset of $X$ and also $Y_c\subset Y$. Since $G[X]$ and $G[Y]$ are  cliques in $G$, $G^c[X_c]$ and $G^c[Y_c]$ are cliques in~$G^c$. In addition, since $G$ is  a reduced $X$-complete graph, $G^c$ is also a reduced $X_c$-complete graph, i.e.~$N_{Y_c}(x_i)\subseteq N_{Y_c}(x_j)$ or $N_{Y_c}(x_j)\subseteq N_{Y_c}(x_i)$ for $x_i,x_j\in X_c$. Moreover, we have $N_{Y_c}(x_i)\not=N_{Y_c}(x_j)$ in $G^c$ since  $G^c$ is  canonical, and thus we may assume that $N_{Y_c}(v_1)\subset N_{Y_c}(x_1)\subset \cdots\subset N_{Y_c}(x_{t_c})$. The latter implies that $N_{X_c}(y_1)\subset N_{X_c}(y_2)\subset \cdots \subset N_{X_c}(y_{k-t_c-1})$. Thus, for $G^c$ we have
\begin{equation}\label{de-eq-2}
0=|N_{Y_c}(v_1)|<|N_{Y_c}(x_1)|< \cdots <|N_{Y_c}(x_{t_c})|\leq |Y_c|=k-t_c-1,
\end{equation}
and
\begin{equation}\label{de-eq-3}
0\leq |N_{X_c}(y_1)|<|N_{X_c}(y_2)|<\cdots < |N_{X_c}(y_{k-t_c-1})|\leq |X_c|= t_c.
\end{equation}
From Eq.~(\ref{de-eq-2}), we have $t_c\leq |N_{Y_c}(x_{t_c})|\leq k-t_c-1$. Similarly, $k-t_c-2\leq |N_{X_c}(y_{k-t_c-1})|\leq t_c$ follows from Eq.~(\ref{de-eq-3}). Thus $k-2\leq 2t_c\leq k-1$, and then $t_c =\lceil\frac{k}{2}\rceil-1$ and $k-t_c-1=\lfloor\frac{k}{2}\rfloor$.

If $k$ is even, then $t_c=\frac{k}{2}-1$. From Eq.~(\ref{de-eq-3}), we have $|N_{X_c}(y_i)|=i-1$ for $1\leq i\leq k-t_c-1=\frac{k}{2}$. Thus we may assume that
\begin{equation*}
\begin{aligned}
&N_{X_c}(y_1)=\emptyset, N_{X_c}(y_2)=\{x_{\frac{k}{2}-1}\}, \ldots, N_{X_c}(y_{\frac{k}{2}-1})=\{x_{\frac{k}{2}-1}, x_{\frac{k}{2}-2} \ldots, x_{2}\}, \\
&N_{X_c}(y_{\frac{k}{2}})=\{x_{\frac{k}{2}-1}, x_{\frac{k}{2}-2}, \ldots, x_{1}\},
\end{aligned}
\end{equation*}
and then $N_{Y_c}(v_1)=\emptyset, N_{Y_c}(x_1)=\{y_{\frac{k}{2}}\}, \cdots, N_{Y_c}(x_{\frac{k}{2}-2})=\{y_{\frac{k}{2}}, y_{\frac{k}{2}-1}, \ldots, y_{3}\},
N_{Y_c}(x_{\frac{k}{2}-1})=\{y_{\frac{k}{2}}, y_{\frac{k}{2}-1}, \ldots, y_{2} \}.$
It follows that $G^c\cong G_k$.

In a very similar way we verify that $G^c\cong G_k$ if $k$ is odd.

It remains to consider the parameter $k$. Recall that $G\in \mathcal{T}_n^s,~ 2\leq s\leq n-3$, and so we have $p(G)=2$ and $\lambda_3(G)=0$. If $k\leq 2$, then $G\cong G_k[K_{n_1}, K_{n_2},\ldots,K_{n_k}]$ is a complete graph which is impossible, and so $k\ge 3$. If $k=3$, then  $\lambda_3(G)<0$ by Lemma \ref{lem-4-7}. Hence, $k\geq 4$. On the other hand, since $G^c\cong G_k$ is an induced subgraph of $G$, we have $\lambda_3(G_k)\leq \lambda_3(G)=0$. If $k\geq 15$, then $\lambda_3(G_k)\geq \lambda_3(G_{15})$ (this holds because $G_{15}$ is an induced subgraph of $G_k$; see the discussion in Section~\ref{sec:prel}). On the other hand, we compute $\lambda_3(G_{15})= 0.1358>0$. This implies that $k\leq 14$, and we are done.
\end{proof}

Lemma \ref{lem-4-6} gives the structure of every reduced $X$-complete graph. However, not every graph with this structure is reduced $X$-complete. In other words, not every such a graph has exactly two positive eigenvalues. In what follows, we determine the parameters $k, n_1, n_2, \ldots, n_k$ for which $G_k[K_{n_1}, K_{n_2}, \ldots, K_{n_k}]$ is $X$-complete and reduced. It is convenient to label
the vertices of $G_k$ by $V(G_{2r})=\{v_1, v_2,\ldots,v_r, w_1, w_2,\ldots, w_r\}$ for $k=2r$ and $V(G_{2r+1})=\{v_1, v_2, \ldots, v_r, w_1, w_2, \ldots, w_r, v_{r+1}\}$ for $k=2r+1$, as in Figure~\ref{fig-1}. To ease language, we follow \cite{M.Petrovi'c} and abbreviate $G_{2r}[K_{n_1}, K_{n_2} \ldots, $ $K_{n_r},K_{n_{r+1}},\ldots, K_{n_{2r}}]$ to $B_{2r}(n_1, n_2, \ldots, n_r; n_{r+1}, n_{r+2}, \ldots, n_{2r})$ and $G_{2r+1}[K_{n_1}, K_{n_2}, \ldots,K_{n_r},K_{n_{r+1}}, \ldots,$ $K_{n_{2r}}, K_{n_{2r+1}}]$ to $B_{2r+1}$ $(n_1, n_2, \ldots, n_r; n_{r+1}, n_{r+2} \ldots, n_{2r}; n_{2r+1})$. One may observe that $$B_{2r}(n_1, n_2, \ldots,n_r; n_{r+1}, n_{r+2}, \ldots, n_{2r})
\cong B_{2r}(n_{r+1}, n_{r+2}, \ldots,n_{2r};n_1, n_2, \ldots, n_r),$$ and similarly for $B_{2r+1}$. In relation to this, we will always use the former notation for isomorphic graphs. We also set
\begin{equation}\label{eq:B} \mathcal{B}^k_n=\{G\cong B_k(n_1, n_2,\ldots, n_k)\mid 4\leq k\leq 14, n_1+ n_2+\cdots +n_k=n, n_i\geq 1\},\end{equation}
where for the sake of simplicity the semicolons are removed from the notation for $B_k$.

For every $G\in\mathcal{B}^k_n$, we have $\lambda_2(G)>0$ (say, because $G$ contains $P_4$ as an induced subgraph). To extract reduced $X$-complete graphs of $\mathcal{B}^k_n$, in the forthcoming Lemma~\ref{lem-4-12} we show that each of them has at most 14 vertices, which of course implies that there is a finite number of such graphs. In the next step we determine all of them by the computer search. The result of Lemma~\ref{lem-4-12} relies on Lemmas~\ref{lem-1} and~\ref{lem-4-11} which are known from the previous works.

\begin{lem}\cite{F.Duan}\label{lem-1}
	Let $G\in \mathcal{B}^k_n$, where $4 \le k\le 9$ and $n\geq 14$. If $\lambda_3(G)\geq 0$, then $G$ contains an induced subgraph $\Gamma\in \mathcal{B}^k_{14}$ with $\lambda_3(\Gamma)\geq 0$.
\end{lem}

\begin{cor}\label{cor-1}
	Let $G\in \mathcal{B}^k_n$, where $4 \le k\le 9$ and $n\geq 15$. If $\lambda_3(G)\geq 0$, then $G$ contains an induced subgraph $\Gamma\in \mathcal{B}^k_{15}$ with $\lambda_3(\Gamma)\geq 0$.
\end{cor}
\begin{proof}
	According to Lemma \ref{lem-1}, $G$ contains an induced subgraph $\Gamma'\in \mathcal{B}^k_{14}$ with $\lambda_3(\Gamma')\geq 0$ since $\lambda_3(G)\geq 0$. By adding a vertex to $\Gamma'$, we easily get an induced subgraph $\Gamma\in \mathcal{B}^k_{15}$ of $G$ with $\lambda_3(\Gamma)\geq \lambda_3(\Gamma')\geq 0$.
\end{proof}

\begin{lem}\cite{M.R.Oboudi3}\label{lem-4-11}
If $G\in \mathcal{B}_n^k$ for $10\leq k\leq 14$ and $n\geq 14$, then $\lambda_3(G)\geq 0$.
\end{lem}

We use the computer search to obtain the following lemma. Namely, for $4 \le k\le 14$, a $k$-partition of $15$ gives a solution $(n_1,n_2,\ldots,n_k)$ for the equation $n_1+n_2+\cdots+n_k=15$. Such a solution determines the graph $G\cong B_k(n_1,n_2,\ldots,n_k)\in \mathcal{B}_{15}^k$. We consider all the possibilities and verify that for each of them $\lambda_3(G)\neq 0$ holds.

\begin{lem}\label{lem-4-9}
Let $G$ be a graph of $\mathcal{B}_{15}^k$, for $4 \le k\le 14$. Then $\lambda_3(G) \neq 0$.
\end{lem}

We are ready to prove the announced lemma.

\begin{lem}\label{lem-4-12}
There are no reduced $X$-complete graphs in $\mathcal{T}_n^s$ for $n\geq 15$ and $2\leq s\leq n-3$.
\end{lem}
\begin{proof}
By  way of contradiction, let $G$ be a reduced $X$-complete graph belonging to $\mathcal{T}_n^s$, where $n\geq 15$ and $2\leq s\leq n-3$. Then $\lambda_3(G)=\lambda_4(G)=0$, and $G\in \mathcal{B}_n^k, 4\leq k\leq 14$, by Lemma~\ref{lem-4-6}.

We first assume that $4\le k\le 9$. Since $\lambda_3(G)=0$, $G$ contains an induced subgraph $\Gamma\in \mathcal{B}_{15}^k$ with $\lambda_3(\Gamma)\ge 0$ by Corollary~\ref{cor-1}. Moreover, we have $\lambda_3(\Gamma)=0$, since for otherwise $0<\lambda_3(\Gamma)\le \lambda_3(G)$. But $\lambda_3(\Gamma)=0$ contradicts Lemma~\ref{lem-4-9}.

We next assume that $10\le k\le 14$. By deleting $n-15$ vertices of $G$, we obtain an induced subgraph $\Gamma\in \mathcal{B}^k_{15}$ with $\lambda_3(\Gamma)\geq 0$ by Lemma \ref{lem-4-11}, and then $\lambda_3(\Gamma)=0$ as before, which again contradicts Lemma \ref{lem-4-9}, and we are done.
\end{proof}

According to Lemma \ref{lem-4-12}, the number of reduced $X$-complete graphs is finite and we can determine all of them by computer search since each of them belongs to \eqref{eq:B}, i.e.~it has a particular structure. In this way we arrive at the following result.

\begin{lem}\label{lem-4-13}
There are exactly $175$ reduced $X$-complete graphs. They are listed in Table~\ref{tab-3}.
\end{lem}

We see from Table~\ref{tab-3} that every graph of Lemma~\ref{lem-4-13} has 14 vertices and exactly two zero eigenvalues (so, $s=2$ holds). To give a proper insight, we illustrate those with $k=6$ in Figure~\ref{fig-6}.

\begin{table}
\caption{\scriptsize{Reduced $X$-complete graphs $B_k(n_1, n_2, \ldots,n_k)$}.}
\centering
\tiny
\begin{tabular}{|clc|}
\hline\\[-1.5mm]
$k$ & Graphs & $\#$ \\[0.7mm]
\hline\\[-1.8mm]
$4$&  & 0\\[0.5mm]
\hline\\[-1.5mm]
$5$&  & 0\\
\hline\\[-1.5mm]
$6$ &\tabincell{l}{$B_6(4,3,3; 2,1,1)$, $B_6(3,2,4; 2,1,2)$, $B_6(5,2,2; 2,2,1)$, $B_6(3,1,3; 2,2,3)$, $B_6(4,1,2; 2,3,2)$, \\[0.7mm]
$B_6(3,4,2; 3,1,1)$,$B_6(4,2,1; 3,3,1)$.} & 7\\[0.7mm]
\hline\\[-1.5mm]
$7$ &\tabincell{l}{
$B_7(3,3,4; 1,1,1; 1)$, $B_7(2,4,3; 1,1,1; 2)$, $B_7(2,2,5; 1,1,2; 1)$, $B_7(1,3,3; 1,1,2; 3)$,\\[0.7mm]
$B_7(1,2,4; 1,1,3; 2)$, $B_7(4,2,3; 1,2,1; 1)$, $B_7(2,3,2; 1,2,1; 3)$, $B_7(1,2,2; 1,2,2; 4)$,\\[0.7mm]
$B_7(2,1,4; 1,2,3; 1)$, $B_7(3,2,2; 1,3,1; 2)$, $B_7(3,1,3; 1,3,2; 1)$, $B_7(2,1,2; 1,4,2; 2)$,\\[0.7mm]
$B_7(2,5,2; 2,1,1; 1)$, $B_7(2,3,1; 2,3,1; 2)$, $B_7(3,2,1; 2,4,1; 1)$.}& 15\\[0.7mm]
\hline\\[-1.5mm]
$8$ &\tabincell{l}{
$B_8(1,3,3,3; 1,1,1,1)$, $B_8(2,3,3,2; 1,1,1,1)$, $B_8(3,3,3,1; 1,1,1,1)$, $B_8(1,2,4,2; 1,1,1,2)$,\\[0.7mm]
$B_8(2,2,4,1; 1,1,1,2)$, $B_8(1,2,2,4; 1,1,2,1)$, $B_8(2,2,2,3; 1,1,2,1)$, $B_8(3,2,2,2; 1,1,2,1)$,\\[0.7mm]
$B_8(4,2,2,1; 1,1,2,1)$, $B_8(1,1,3,2; 1,1,2,3)$, $B_8(2,1,3,1; 1,1,2,3)$, $B_8(2,1,2,2; 1,1,3,2)$,\\[0.7mm]
$B_8(3,1,2,1; 1,1,3,2)$, $B_8(1,4,2,2; 1,2,1,1)$, $B_8(2,4,2,1; 1,2,1,1)$, $B_8(1,2,3,1; 1,2,1,3)$,\\[0.7mm]
$B_8(2,2,1,2; 1,2,3,1)$, $B_8(3,2,1,1; 1,2,3,1)$, $B_8(2,1,1,1; 1,2,4,2)$, $B_8(1,3,2,1; 1,3,1,2)$,\\[0.7mm]
$B_8(2,3,1,1; 1,3,2,1)$, $B_8(2,2,3,2; 2,1,1,1)$, $B_8(3,1,2,3; 2,1,1,1)$, $B_8(3,2,2,2; 2,1,1,1)$,\\[0.7mm]
$B_8(3,3,2,1; 2,1,1,1)$, $B_8(4,2,1,2; 2,1,1,1)$, $B_8(2,2,3,1; 2,1,1,2)$, $B_8(3,1,1,3; 2,1,1,2)$,\\[0.7mm]
$B_8(2,1,2,3; 2,1,2,1)$, $B_8(4,2,1,1; 2,1,2,1)$, $B_8(2,1,2,2; 2,1,2,2)$, $B_8(3,1,1,2; 2,1,2,2)$,\\[0.7mm]
$B_8(3,1,1,1; 2,1,3,2)$, $B_8(2,3,2,1; 2,2,1,1)$, $B_8(4,1,1,2; 2,2,1,1)$, $B_8(3,1,1,2; 2,2,2,1)$,\\[0.7mm]
$B_8(3,2,1,1; 2,2,2,1)$, $B_8(3,3,1,1; 3,1,1,1)$, $B_8(3,2,1,1; 3,2,1,1)$.}& 39\\
\hline\\[-1.5mm]
$9$ &\tabincell{l}{
$B_9(1,2,3,3; 1,1,1,1; 1)$, $B_9(2,1,2,4; 1,1,1,1; 1)$, $B_9(2,2,2,3; 1,1,1,1; 1)$, $B_9(2,3,2,2; 1,1,1,1; 1)$,\\[0.7mm]
$B_9(2,4,2,1; 1,1,1,1; 1)$, $B_9(3,2,1,3; 1,1,1,1; 1)$, $B_9(1,1,3,3; 1,1,1,1; 2)$, $B_9(2,3,1,2; 1,1,1,1; 2)$,\\[0.7mm]
$B_9(1,2,3,2; 1,1,1,2; 1)$, $B_9(2,1,1,4; 1,1,1,2; 1)$, $B_9(1,1,2,3; 1,1,1,2; 2)$, $B_9(1,2,2,2; 1,1,1,2; 2)$,\\[0.7mm]
$B_9(1,3,2,1; 1,1,1,2; 2)$, $B_9(1,2,1,2; 1,1,1,2; 3)$, $B_9(1,2,3,1; 1,1,1,3; 1)$, $B_9(1,1,1,3; 1,1,1,3; 2)$,\\[0.7mm]
$B_9(1,1,2,4; 1,1,2,1; 1)$, $B_9(3,2,1,2; 1,1,2,1; 1)$, $B_9(2,3,1,1; 1,1,2,1; 2)$, $B_9(1,1,2,3; 1,1,2,2; 1)$,\\[0.7mm]
$B_9(2,1,1,3; 1,1,2,2; 1)$, $B_9(1,2,1,1; 1,1,2,2; 3)$, $B_9(3,2,1,1; 1,1,3,1; 1)$, $B_9(2,1,1,2; 1,1,3,2; 1)$,\\[0.7mm]
$B_9(2,1,1,1; 1,1,4,2; 1)$, $B_9(1,3,2,2; 1,2,1,1; 1)$, $B_9(3,1,1,3; 1,2,1,1; 1)$, $B_9(1,2,2,2; 1,2,1,1; 2)$,\\[0.7mm]
$B_9(2,2,1,2; 1,2,1,1; 2)$, $B_9(1,3,2,1; 1,2,1,2; 1)$, $B_9(2,1,1,3; 1,2,2,1; 1)$, $B_9(2,2,1,2; 1,2,2,1; 1)$,\\[0.7mm]
$B_9(2,3,1,1; 1,2,2,1; 1)$, $B_9(2,1,1,2; 1,3,1,1; 2)$, $B_9(2,4,1,1; 2,1,1,1; 1)$, $B_9(2,3,1,1; 2,2,1,1; 1)$.} & 36\\
\hline\\[-1.5mm]
$10$ &\tabincell{l}{
$B_{10}(1,1,2,3,2; 1,1,1,1,1)$, $B_{10}(1,2,1,2,3; 1,1,1,1,1)$, $B_{10}(1,2,2,2,2; 1,1,1,1,1)$,\\[0.7mm]
$B_{10}(1,2,3,2,1; 1,1,1,1,1)$, $B_{10}(1,3,2,1,2; 1,1,1,1,1)$, $B_{10}(2,1,2,3,1; 1,1,1,1,1)$,\\[0.7mm]
$B_{10}(2,2,1,2,2; 1,1,1,1,1)$, $B_{10}(2,2,2,2,1; 1,1,1,1,1)$, $B_{10}(2,3,2,1,1; 1,1,1,1,1)$,\\[0.7mm]
$B_{10}(3,2,1,2,1; 1,1,1,1,1)$, $B_{10}(1,1,1,3,2; 1,1,1,1,2)$, $B_{10}(1,2,3,1,1; 1,1,1,1,2)$,\\[0.7mm]
$B_{10}(2,1,1,3,1; 1,1,1,1,2)$, $B_{10}(1,1,2,3,1; 1,1,1,2,1)$, $B_{10}(1,2,1,1,3; 1,1,1,2,1)$,\\[0.7mm]
$B_{10}(2,2,1,1,2; 1,1,1,2,1)$, $B_{10}(3,2,1,1,1; 1,1,1,2,1)$, $B_{10}(1,1,1,2,2; 1,1,1,2,2)$,\\[0.7mm]
$B_{10}(1,1,2,2,1; 1,1,1,2,2)$, $B_{10}(2,1,1,2,1; 1,1,1,2,2)$, $B_{10}(1,1,2,1,1; 1,1,1,2,3)$,\\[0.7mm]
$B_{10}(2,1,1,1,1; 1,1,1,3,2)$, $B_{10}(1,3,2,1,1; 1,1,2,1,1)$, $B_{10}(2,1,1,2,2; 1,1,2,1,1)$,\\[0.7mm]
$B_{10}(3,1,1,2,1; 1,1,2,1,1)$, $B_{10}(1,2,1,1,2; 1,1,2,2,1)$, $B_{10}(2,1,1,2,1; 1,1,2,2,1)$,\\[0.7mm]
$B_{10}(2,2,1,1,1; 1,1,2,2,1)$, $B_{10}(1,2,1,1,1; 1,1,3,2,1)$, $B_{10}(1,3,1,1,2; 1,2,1,1,1)$,\\[0.7mm]
$B_{10}(2,3,1,1,1; 1,2,1,1,1)$, $B_{10}(1,2,2,1,1; 1,2,1,1,2)$, $B_{10}(1,2,2,1,1; 1,2,2,1,1)$,\\[0.7mm]
$B_{10}(2,2,1,1,1; 1,2,2,1,1)$, $B_{10}(2,1,1,1,1; 1,2,3,1,1)$, $B_{10}(2,1,1,2,2; 2,1,1,1,1)$,\\[0.7mm]
$B_{10}(2,1,2,2,1; 2,1,1,1,1)$, $B_{10}(2,2,2,1,1; 2,1,1,1,1)$, $B_{10}(3,1,1,1,2; 2,1,1,1,1)$,\\[0.7mm]
$B_{10}(3,2,1,1,1; 2,1,1,1,1)$, $B_{10}(2,1,1,2,1; 2,1,1,1,2)$, $B_{10}(3,1,1,1,1; 2,1,2,1,1)$,\\[0.7mm]
$B_{10}(2,2,1,1,1; 2,2,1,1,1)$.} & 43\\
\hline\\[-1.5mm]
$11$& \tabincell{l}{
$B_{11}(1,1,1,2,3; 1,1,1,1,1;1)$, $B_{11}(1,1,2,2,2; 1,1,1,1,1;1)$; $B_{11}(1,1,3,2,1; 1,1,1,1,1;1)$,\\[0.7mm]
$B_{11}(1,2,2,1,2; 1,1,1,1,1;1)$, $B_{11}(2,1,1,1,3; 1,1,1,1,1;1)$, $B_{11}(2,2,1,1,2; 1,1,1,1,1;1)$,\\[0.7mm]
$B_{11}(2,3,1,1,1; 1,1,1,1,1;1)$, $B_{11}(1,1,2,1,2; 1,1,1,1,1;2)$, $B_{11}(1,1,1,2,2; 1,1,1,1,2;1)$,\\[0.7mm]
$B_{11}(1,1,2,2,1; 1,1,1,1,2;1)$, $B_{11}(1,2,2,1,1; 1,1,1,1,2;1)$, $B_{11}(1,1,1,1,2; 1,1,1,1,2;2)$,\\[0.7mm]
$B_{11}(1,2,1,1,1; 1,1,1,1,2;2)$, $B_{11}(1,1,1,2,1; 1,1,1,1,3;1)$, $B_{11}(1,2,2,1,1; 1,1,2,1,1;1)$,\\[0.7mm]
$B_{11}(2,1,1,1,2; 1,1,2,1,1;1)$, $B_{11}(2,2,1,1,1; 1,1,2,1,1;1)$, $B_{11}(1,1,2,1,1; 1,1,2,1,1;2)$,\\[0.7mm]
$B_{11}(2,1,1,1,1; 1,1,3,1,1;1)$, $B_{11}(1,2,1,1,2; 1,2,1,1,1;1)$.}& 20\\
\hline\\[-1.5mm]
$12$ &\tabincell{l}{
$B_{12}(1,1,1,1,2,2;1,1,1,1,1,1)$, $B_{12}(1,1,1,2,2,1;1,1,1,1,1,1)$, $B_{12}(1,1,2,2,1,1;1,1,1,1,1,1)$,\\[0.7mm]
$B_{12}(1,2,1,1,1,2;1,1,1,1,1,1)$, $B_{12}(1,2,2,1,1,1;1,1,1,1,1,1)$, $B_{12}(2,1,1,1,2,1;1,1,1,1,1,1)$,\\[0.7mm]
$B_{12}(2,2,1,1,1,1;1,1,1,1,1,1)$, $B_{12}(1,1,1,2,1,1;1,1,1,1,1,2)$, $B_{12}(1,1,1,1,2,1;1,1,1,1,2,1)$,\\[0.7mm]
$B_{12}(2,1,1,1,1,1;1,1,1,2,1,1)$, $B_{12}(1,2,1,1,1,1;1,1,2,1,1,1)$, $B_{12}(2,1,1,1,1,1;2,1,1,1,1,1)$.}& 12\\[0.7mm]
\hline\\[-1.5mm]
$13$ & $B_{13}(1,1,1,1,1,2;1,1,1,1,1,1;1)$, $B_{13}(1,1,2,1,1,1;1,1,1,1,1,1;1)$.& 2\\[0.7mm]
\hline\\[-1.5mm]
$14$ & $B_{14}(1,1,1,1,1,1,1;1,1,1,1,1,1,1)$.& 1\\[0.7mm]

\hline
\end{tabular}\label{tab-3}
\end{table}

\begin{figure}[t]
\begin{center}
	\unitlength 1.4mm % = 2.845pt
	\linethickness{0.4pt}
	\ifx\plotpoint\undefined\newsavebox{\plotpoint}\fi % GNUPLOT compatibility
	\setlength{\unitlength}{3.3pt}
	\begin{picture}(125.5,61.25)(0,0)
		\put(14.25,56.75){\oval(21.25,9)[]}
		\put(7.242,57.128){\circle{5.483}}
		\put(14.25,44.75){\oval(21.25,9)[]}
		\put(7.242,45.128){\circle{5.483}}
		\put(5.718,44.851){\circle*{.94}}
		\put(8.843,44.851){\circle*{.94}}
		\put(21.5,45.25){\circle{4.25}}
		\put(21.47,45.351){\circle*{.94}}
		\put(14.5,43.5){\circle{4.25}}
		\put(14.47,43.601){\circle*{.94}}
		\put(5.72,56.47){\circle*{.94}}
		\put(8.845,56.47){\circle*{.94}}
		\put(5.845,58.47){\circle*{.94}}
		\put(8.97,58.47){\circle*{.94}}
		\put(12.72,57.345){\circle*{.94}}
		\put(21.117,56.617){\circle{5.483}}
		\put(21.219,57.97){\circle*{.94}}
		\put(19.594,55.595){\circle*{.94}}
		\put(22.719,55.595){\circle*{.94}}
		\put(14.367,58.367){\circle{5.483}}
		\put(16.22,58.97){\circle*{.94}}
		%\emline(22.75,55.625)(21.25,45)
		\multiput(22.75,55.625)(-.03333333,-.23611111){45}{\line(0,-1){.23611111}}
		%\end
		%\emline(19.375,55.625)(21.5,45.125)
		\multiput(19.375,55.625)(.03373016,-.16666667){63}{\line(0,-1){.16666667}}
		%\end
		\put(21.25,58){\line(0,-1){12.5}}
		%\emline(19.625,55.625)(14.5,43.625)
		\multiput(19.625,55.625)(-.033717105,-.078947368){152}{\line(0,-1){.078947368}}
		%\end
		%\emline(22.75,55.625)(14.5,43.375)
		\multiput(22.75,55.625)(-.0336734694,-.05){245}{\line(0,-1){.05}}
		%\end
		\qbezier(21.125,58)(15.188,55.688)(14.5,43.625)
		\put(.25,57){\makebox(0,0)[cc]{\footnotesize$K_{10}$}}
		\put(0,43.625){\makebox(0,0)[cc]{\footnotesize$K_4$}}
		\put(13.97,60.22){\circle*{.94}}
		%\emline(12.625,57.375)(21.375,45.375)
		\multiput(12.625,57.375)(.0336538462,-.0461538462){260}{\line(0,-1){.0461538462}}
		%\end
		%\emline(16.25,59)(21.375,45.5)
		\multiput(16.25,59)(.033717105,-.088815789){152}{\line(0,-1){.088815789}}
		%\end
		%\emline(13.875,60.25)(21.375,45.5)
		\multiput(13.875,60.25)(.033632287,-.066143498){223}{\line(0,-1){.066143498}}
		%\end
		\put(14.5,36.25){\makebox(0,0)[cc]{\scriptsize$B_6(4,3,3;2,1,1)$}}
		\put(13.994,20.326){\oval(21.25,9)[]}
		\put(13.994,8.326){\oval(21.25,9)[]}
		\put(8.111,21.568){\circle{5.483}}
		\put(6.589,20.91){\circle*{.94}}
		\put(9.714,20.91){\circle*{.94}}
		\put(6.714,22.911){\circle*{.94}}
		\put(9.839,22.911){\circle*{.94}}
		\put(20.736,21.943){\circle{5.483}}
		\put(22.588,22.546){\circle*{.94}}
		\put(6.994,8.826){\circle{4.25}}
		%\emline(6.994,8.826)(6.869,8.951)
		\multiput(6.994,8.826)(-.03125,.03125){4}{\line(0,1){.03125}}
		%\end
		\put(5.964,8.796){\circle*{.94}}
		\put(8.214,8.671){\circle*{.94}}
		\put(14.486,7.818){\circle{5.483}}
		\put(13.744,19.576){\circle{4.25}}
		\put(13.714,19.677){\circle*{.94}}
		\put(14.714,9.171){\circle*{.94}}
		\put(13.214,6.796){\circle*{.94}}
		\put(16.089,6.921){\circle*{.94}}
		\put(21.361,9.068){\circle{5.483}}
		\put(19.837,8.789){\circle*{.94}}
		\put(22.962,8.789){\circle*{.94}}
		%\emline(13.619,19.701)(19.744,8.951)
		\multiput(13.619,19.701)(.033653846,-.059065934){182}{\line(0,-1){.059065934}}
		%\end
		%\emline(13.744,19.451)(22.994,8.951)
		\multiput(13.744,19.451)(.0336363636,-.0381818182){275}{\line(0,-1){.0381818182}}
		%\end
		\put(19.714,20.796){\circle*{.94}}
		%\emline(19.619,20.826)(14.744,8.951)
		\multiput(19.619,20.826)(-.03362069,-.081896552){145}{\line(0,-1){.081896552}}
		%\end
		%\emline(19.744,20.826)(16.119,7.201)
		\multiput(19.744,20.826)(-.033564815,-.126157407){108}{\line(0,-1){.126157407}}
		%\end
		\qbezier(19.619,20.826)(14.994,17.826)(13.369,6.826)
		\put(19.744,20.826){\line(0,-1){12}}
		%\emline(19.619,20.701)(22.994,9.076)
		\multiput(19.619,20.701)(.033415842,-.11509901){101}{\line(0,-1){.11509901}}
		%\end
		%\emline(22.619,22.576)(19.869,8.951)
		\multiput(22.619,22.576)(-.03353659,-.16615854){82}{\line(0,-1){.16615854}}
		%\end
		\put(22.744,22.701){\line(0,-1){13.875}}
		%\emline(22.619,22.576)(14.869,9.076)
		\multiput(22.619,22.576)(-.033695652,-.058695652){230}{\line(0,-1){.058695652}}
		%\end
		%\emline(22.619,22.576)(16.119,6.826)
		\multiput(22.619,22.576)(-.033678756,-.081606218){193}{\line(0,-1){.081606218}}
		%\end
		\qbezier(22.619,22.701)(9.369,13.701)(13.119,6.701)
		\put(.494,21.076){\makebox(0,0)[cc]{\footnotesize$K_7$}}
		\put(.244,7.326){\makebox(0,0)[cc]{\footnotesize$K_7$}}
		\put(14.369,.201){\makebox(0,0)[cc]{\scriptsize$B_6(4,1,2;2,3,2)$}}
		\put(48,56.625){\oval(21.25,9)[]}
		\put(48,44.625){\oval(21.25,9)[]}
		\put(40.992,45.003){\circle{5.483}}
		\put(39.468,44.726){\circle*{.94}}
		\put(42.593,44.726){\circle*{.94}}
		\put(55.25,45.125){\circle{4.25}}
		\put(48.25,43.375){\circle{4.25}}
		\put(48.22,43.476){\circle*{.94}}
		\put(46.47,57.22){\circle*{.94}}
		\put(48.117,58.242){\circle{5.483}}
		\put(41.617,56.617){\circle{5.483}}
		\put(41.717,57.97){\circle*{.94}}
		\put(40.092,55.595){\circle*{.94}}
		\put(43.217,55.595){\circle*{.94}}
		\put(55.47,43.845){\circle*{.94}}
		\put(54.867,56.491){\circle{5.483}}
		\put(53.344,55.833){\circle*{.94}}
		\put(56.47,55.833){\circle*{.94}}
		\put(53.469,57.834){\circle*{.94}}
		\put(56.595,57.834){\circle*{.94}}
		\put(49.595,58.595){\circle*{.94}}
		\put(54.72,46.22){\circle*{.94}}
		\put(35,56.75){\makebox(0,0)[cc]{\footnotesize$K_9$}}
		\put(34.625,44.375){\makebox(0,0)[cc]{\footnotesize$K_5$}}
		%\emline(46.375,57.125)(54.625,46.375)
		\multiput(46.375,57.125)(.0336734694,-.043877551){245}{\line(0,-1){.043877551}}
		%\end
		%\emline(46.25,57.125)(55.375,43.875)
		\multiput(46.25,57.125)(.0336715867,-.0488929889){271}{\line(0,-1){.0488929889}}
		%\end
		%\emline(49.625,58.625)(54.75,46.25)
		\multiput(49.625,58.625)(.033717105,-.081414474){152}{\line(0,-1){.081414474}}
		%\end
		%\emline(49.5,58.75)(55.25,43.75)
		\multiput(49.5,58.75)(.033625731,-.087719298){171}{\line(0,-1){.087719298}}
		%\end
		%\emline(53.25,55.75)(48.375,43.5)
		\multiput(53.25,55.75)(-.03362069,-.084482759){145}{\line(0,-1){.084482759}}
		%\end
		%\emline(56.5,56)(48.375,43.5)
		\multiput(56.5,56)(-.033713693,-.05186722){241}{\line(0,-1){.05186722}}
		%\end
		%\emline(56.75,57.875)(48.25,43.625)
		\multiput(56.75,57.875)(-.0337301587,-.056547619){252}{\line(0,-1){.056547619}}
		%\end
		\qbezier(53.5,57.875)(50.438,55.813)(48.125,43.5)
		%\emline(53.25,56)(54.75,46.125)
		\multiput(53.25,56)(.03333333,-.21944444){45}{\line(0,-1){.21944444}}
		%\end
		%\emline(56.375,56)(54.625,46.25)
		\multiput(56.375,56)(-.03365385,-.1875){52}{\line(0,-1){.1875}}
		%\end
		%\emline(56.375,55.75)(55.5,44)
		\multiput(56.375,55.75)(-.0336538,-.4519231){26}{\line(0,-1){.4519231}}
		%\end
		\qbezier(56.625,57.875)(59.25,51.438)(55.375,43.75)
		\qbezier(53.5,57.75)(55.063,53.063)(55.375,44.125)
		\put(48.875,36.125){\makebox(0,0)[cc]{\scriptsize$B_6(3,2,4;2,1,2)$}}
		\put(69.125,56.875){\makebox(0,0)[cc]{\footnotesize$K_9$}}
		\put(68.75,44.5){\makebox(0,0)[cc]{\footnotesize$K_{5}$}}
		\put(81.75,56.75){\oval(21.25,9)[]}
		\put(81.75,44.75){\oval(21.25,9)[]}
		\put(75.867,57.992){\circle{5.483}}
		\put(74.345,57.334){\circle*{.94}}
		\put(77.47,57.334){\circle*{.94}}
		\put(74.47,59.335){\circle*{.94}}
		\put(77.595,59.335){\circle*{.94}}
		\put(76.095,58.345){\circle*{.94}}
		\put(81.867,56.117){\circle{5.483}}
		\put(80.343,55.839){\circle*{.94}}
		\put(83.468,55.839){\circle*{.94}}
		\put(86.844,57.345){\circle*{.94}}
		\put(88.492,58.367){\circle{5.483}}
		\put(90.344,58.97){\circle*{.94}}
		\put(74.75,45.25){\circle{4.25}}
		%\emline(74.75,45.25)(74.625,45.375)
		\multiput(74.75,45.25)(-.03125,.03125){4}{\line(-1,0){.03125}}
		%\end
		\put(73.72,45.22){\circle*{.94}}
		\put(75.97,45.095){\circle*{.94}}
		\put(82.242,44.242){\circle{5.483}}
		\put(81.094,43.844){\circle*{.94}}
		\put(83.969,43.969){\circle*{.94}}
		\put(89.25,46.25){\circle{4.25}}
		\put(89.22,46.351){\circle*{.94}}
		%\emline(80.25,55.75)(89,46.5)
		\multiput(80.25,55.75)(.0336538462,-.0355769231){260}{\line(0,-1){.0355769231}}
		%\end
		%\emline(83.375,55.875)(89.25,46.375)
		\multiput(83.375,55.875)(.033571429,-.054285714){175}{\line(0,-1){.054285714}}
		%\end
		%\emline(86.875,57.375)(81,43.75)
		\multiput(86.875,57.375)(-.033571429,-.077857143){175}{\line(0,-1){.077857143}}
		%\end
		%\emline(87,57.625)(84,44)
		\multiput(87,57.625)(-.03370787,-.15308989){89}{\line(0,-1){.15308989}}
		%\end
		%\emline(86.75,57.5)(89.375,46.25)
		\multiput(86.75,57.5)(.03365385,-.14423077){78}{\line(0,-1){.14423077}}
		%\end
		%\emline(90.375,58.875)(81.125,43.75)
		\multiput(90.375,58.875)(-.0336363636,-.055){275}{\line(0,-1){.055}}
		%\end
		%\emline(90.5,59)(83.875,43.75)
		\multiput(90.5,59)(-.033629442,-.077411168){197}{\line(0,-1){.077411168}}
		%\end
		%\emline(90.375,59)(89.25,46.25)
		\multiput(90.375,59)(-.03308824,-.375){34}{\line(0,-1){.375}}
		%\end
		\put(81.875,36.125){\makebox(0,0)[cc]{\scriptsize$B_6(5,2,2;2,2,1)$}}
		\put(114.875,56.375){\oval(21.25,9)[]}
		\put(114.875,44.375){\oval(21.25,9)[]}
		\put(108.992,57.617){\circle{5.483}}
		\put(107.47,56.959){\circle*{.94}}
		\put(110.595,56.959){\circle*{.94}}
		\put(121.617,57.992){\circle{5.483}}
		\put(123.469,58.595){\circle*{.94}}
		\put(107.875,44.875){\circle{4.25}}
		%\emline(107.875,44.875)(107.75,45)
		\multiput(107.875,44.875)(-.03125,.03125){4}{\line(-1,0){.03125}}
		%\end
		\put(106.845,44.845){\circle*{.94}}
		\put(109.095,44.72){\circle*{.94}}
		\put(115.367,43.867){\circle{5.483}}
		\put(114.219,43.469){\circle*{.94}}
		\put(117.094,43.594){\circle*{.94}}
		\put(109.345,59.22){\circle*{.94}}
		\put(114.875,55.25){\circle{4.25}}
		\put(114.845,55.351){\circle*{.94}}
		\put(122.242,44.867){\circle{5.483}}
		\put(122.468,46.595){\circle*{.94}}
		\put(123.718,44.22){\circle*{.94}}
		%\emline(114.875,55.375)(122.375,46.75)
		\multiput(114.875,55.375)(.033632287,-.03867713){223}{\line(0,-1){.03867713}}
		%\end
		%\emline(114.875,55.375)(123.625,44)
		\multiput(114.875,55.375)(.0336538462,-.04375){260}{\line(0,-1){.04375}}
		%\end
		\put(120.595,59.845){\circle*{.94}}
		%\emline(120.5,59.75)(114.125,43.375)
		\multiput(120.5,59.75)(-.033730159,-.086640212){189}{\line(0,-1){.086640212}}
		%\end
		%\emline(120.625,59.875)(117.125,43.625)
		\multiput(120.625,59.875)(-.033653846,-.15625){104}{\line(0,-1){.15625}}
		%\end
		\put(120.595,44.22){\circle*{.94}}
		\put(120.625,59.875){\line(0,-1){15.875}}
		%\emline(114.75,55.25)(120.5,44.125)
		\multiput(114.75,55.25)(.033625731,-.06505848){171}{\line(0,-1){.06505848}}
		%\end
		\put(123.5,58.625){\line(0,-1){14}}
		%\emline(123.375,58.75)(120.625,44.25)
		\multiput(123.375,58.75)(-.03353659,-.17682927){82}{\line(0,-1){.17682927}}
		%\end
		%\emline(123.375,58.75)(122.5,46.5)
		\multiput(123.375,58.75)(-.0336538,-.4711538){26}{\line(0,-1){.4711538}}
		%\end
		%\emline(123.375,58.875)(114.375,43.625)
		\multiput(123.375,58.875)(-.0337078652,-.0571161049){267}{\line(0,-1){.0571161049}}
		%\end
		%\emline(123.375,58.75)(117.125,43.75)
		\multiput(123.375,58.75)(-.033602151,-.080645161){186}{\line(0,-1){.080645161}}
		%\end
		\put(121.595,57.22){\circle*{.94}}
		%\emline(121.625,57.375)(114.125,43.625)
		\multiput(121.625,57.375)(-.033632287,-.061659193){223}{\line(0,-1){.061659193}}
		%\end
		%\emline(121.625,57.125)(117,43.625)
		\multiput(121.625,57.125)(-.033514493,-.097826087){138}{\line(0,-1){.097826087}}
		%\end
		%\emline(121.625,57)(122.5,46.75)
		\multiput(121.625,57)(.0336538,-.3942308){26}{\line(0,-1){.3942308}}
		%\end
		%\emline(121.5,57.375)(123.75,44.25)
		\multiput(121.5,57.375)(.03358209,-.19589552){67}{\line(0,-1){.19589552}}
		%\end
		%\emline(121.5,57.125)(120.625,44.375)
		\multiput(121.5,57.125)(-.0336538,-.4903846){26}{\line(0,-1){.4903846}}
		%\end
		\put(101.5,56.625){\makebox(0,0)[cc]{\footnotesize$K_7$}}
		\put(101.25,43.25){\makebox(0,0)[cc]{\footnotesize$K_7$}}
		\put(115.875,36.125){\makebox(0,0)[cc]{\scriptsize$B_6(3,1,3;2,2,3)$}}
		\put(47.625,20.375){\oval(21.25,9)[]}
		\put(47.625,8.375){\oval(21.25,9)[]}
		\put(52.719,20.97){\circle*{.94}}
		\put(54.367,21.992){\circle{5.483}}
		\put(56.219,22.595){\circle*{.94}}
		\put(55.125,9.875){\circle{4.25}}
		\put(55.095,9.976){\circle*{.94}}
		\put(40.992,21.367){\circle{5.483}}
		\put(39.844,20.969){\circle*{.94}}
		\put(42.719,21.094){\circle*{.94}}
		\put(41.345,22.845){\circle*{.94}}
		\put(47.992,19.742){\circle{5.483}}
		\put(46.47,19.084){\circle*{.94}}
		\put(49.595,19.084){\circle*{.94}}
		\put(46.595,21.084){\circle*{.94}}
		\put(49.72,21.084){\circle*{.94}}
		\put(41.367,7.992){\circle{5.483}}
		\put(39.843,7.713){\circle*{.94}}
		\put(42.968,7.713){\circle*{.94}}
		\put(41.719,9.345){\circle*{.94}}
		\put(48,7.5){\circle{4.25}}
		%\emline(48,7.5)(47.875,7.625)
		\multiput(48,7.5)(-.03125,.03125){4}{\line(0,1){.03125}}
		%\end
		\put(47.845,7.72){\circle*{.94}}
		%\emline(46.5,18.875)(55.125,9.875)
		\multiput(46.5,18.875)(.0336914063,-.03515625){256}{\line(0,-1){.03515625}}
		%\end
		%\emline(49.625,19.25)(55.125,10)
		\multiput(49.625,19.25)(.033536585,-.056402439){164}{\line(0,-1){.056402439}}
		%\end
		%\emline(46.625,21)(55,10)
		\multiput(46.625,21)(.0336345382,-.0441767068){249}{\line(0,-1){.0441767068}}
		%\end
		%\emline(49.75,21)(55.25,9.875)
		\multiput(49.75,21)(.033536585,-.067835366){164}{\line(0,-1){.067835366}}
		%\end
		%\emline(52.75,21)(55.125,10)
		\multiput(52.75,21)(.0334507,-.15492958){71}{\line(0,-1){.15492958}}
		%\end
		%\emline(56.125,22.625)(55,10.375)
		\multiput(56.125,22.625)(-.03308824,-.36029412){34}{\line(0,-1){.36029412}}
		%\end
		%\emline(52.75,21.125)(48,7.75)
		\multiput(52.75,21.125)(-.033687943,-.094858156){141}{\line(0,-1){.094858156}}
		%\end
		%\emline(56.25,22.625)(47.875,7.75)
		\multiput(56.25,22.625)(-.0336345382,-.0597389558){249}{\line(0,-1){.0597389558}}
		%\end
		\put(34.5,7.375){\makebox(0,0)[cc]{\footnotesize$K_5$}}
		\put(34.25,20.75){\makebox(0,0)[cc]{\footnotesize$K_9$}}
		\put(47.875,0){\makebox(0,0)[cc]{\scriptsize$B_6(3,4,2;3,1,1)$}}
		\put(81.875,20.75){\oval(21.25,9)[]}
		\put(81.875,8.75){\oval(21.25,9)[]}
		\put(75.992,21.992){\circle{5.483}}
		\put(74.47,21.334){\circle*{.94}}
		\put(77.595,21.334){\circle*{.94}}
		\put(74.595,23.335){\circle*{.94}}
		\put(77.72,23.335){\circle*{.94}}
		\put(81.992,20.117){\circle{5.483}}
		\put(80.468,19.839){\circle*{.94}}
		\put(83.593,19.839){\circle*{.94}}
		\put(74.875,9.25){\circle{4.25}}
		%\emline(74.875,9.25)(74.75,9.375)
		\multiput(74.875,9.25)(-.03125,.03125){4}{\line(0,1){.03125}}
		%\end
		\put(82.367,8.242){\circle{5.483}}
		\put(81.219,7.844){\circle*{.94}}
		\put(84.094,7.969){\circle*{.94}}
		\put(89.375,10.25){\circle{4.25}}
		\put(89.345,10.351){\circle*{.94}}
		\put(82.345,9.845){\circle*{.94}}
		\put(89.375,21.625){\circle{4.25}}
		\put(89.345,21.726){\circle*{.94}}
		\put(75.095,10.72){\circle*{.94}}
		\put(73.845,8.845){\circle*{.94}}
		\put(76.095,8.72){\circle*{.94}}
		%\emline(80.375,19.875)(89.375,10.375)
		\multiput(80.375,19.875)(.0337078652,-.0355805243){267}{\line(0,-1){.0355805243}}
		%\end
		%\emline(83.625,19.75)(89.375,10.25)
		\multiput(83.625,19.75)(.033625731,-.055555556){171}{\line(0,-1){.055555556}}
		%\end
		\put(89.25,21.875){\line(0,-1){11.625}}
		%\emline(89.25,21.875)(82.375,9.875)
		\multiput(89.25,21.875)(-.03370098,-.058823529){204}{\line(0,-1){.058823529}}
		%\end
		%\emline(89.25,21.875)(84.125,7.875)
		\multiput(89.25,21.875)(-.033717105,-.092105263){152}{\line(0,-1){.092105263}}
		%\end
		\qbezier(89.25,21.875)(84.563,19.375)(81.125,7.875)
		\put(69,21.125){\makebox(0,0)[cc]{\footnotesize$K_7$}}
		\put(69,7.625){\makebox(0,0)[cc]{\footnotesize$K_7$}}
		\put(82,.25){\makebox(0,0)[cc]{\scriptsize$B_6(4,2,1;3,3,1)$}}
	\end{picture}
\end{center}
\caption{The graphs with $k=6$ of Table~\ref{tab-3}.}\label{fig-6}
\end{figure}

At the end of this section, we complete the proof of Theorem~\ref{thm-1-2}.

\medskip

{\it Proof of Theorem \ref{thm-1-2}(b).} According to Lemma~\ref{lem-4-1}, we have partitioned graphs under consideration into the classes $\mathcal{C}_1-\mathcal{C}_3$. In Lemmas~\ref{lem-4-2} and~\ref{lem-4-3} we proved that those of $\mathcal{C}_1\cup \mathcal{C}_2$ belong to $\mathcal{T}_n^s(I)\cup \mathcal{T}_n^s(II)$. Lemma~\ref{lem-4-4} shows that $X$-incomplete graphs of $\mathcal{C}_3$ belong to $\mathcal{T}_n^s(I)\cup \mathcal{T}_n^s(III)$ and Lemma~\ref{lem-4-5} shows that unreduced $X$-complete graphs of $\mathcal{C}_3$ belong to $\mathcal{T}_n^s(III)$. Finally, Lemma~\ref{lem-4-13} completes this proof by showing that there are exactly 175 reduced $X$-complete graphs of $\mathcal{C}_3$ and that $s=2$ holds for each of them.\qed

\section{Conclusion}\label{sec:rec}

In this section we give a short recapitulation of the previous results along with an example. Recall that the collection of graphs with exactly two positive eigenvalues is denoted by $\mathcal{T}_n=\bigcup_{s=0}^{n-3}\mathcal{T}^s_n$.

A complete characterization of graphs belonging to $\mathcal{T}_n$ is given in Theorems~\ref{thm-2-3}, \ref{thm-3-1} and~\ref{thm-1-2}. The first two results are proved in \cite{M.R.Oboudi3} and \cite{F.Duan}, respectively, while the third one is our contribution. We have seen that disconnected graphs of $\mathcal{T}_n$ are listed explicitly in the previous statements, while   connected ones are partitioned into several infinite families (also listed in the statements) and a finite number of the additional individual graphs. Individual graphs are given in \cite{M.R.Oboudi3} (the 601 graphs of $T_n^0$),~\cite{F.Duan} (the 802 graphs of $T_n^1$) and Table~\ref{tab-3} (the 175 graphs of $T_n^2$).  For $s\geq 2$, infinite families of connected graphs consist of the so-called derivable graphs obtained by adding congruent vertices of type I, II or III. Apart from them, there are exactly 175 additional individual graphs (also called underivable graphs and for each of them we have $s=2$).

We conclude the section with an example.

\begin{exa}\label{exa:e}
In Figure~\ref{fig-7} we illustrate the four constructions by taking a graph $F\in \mathcal{T}_{n-1}^{s-1}$ and then adding a congruent vertex $u$ of type I, II, III and IV, respectively. In this way we obtain graphs $F_1\in \mathcal{T}_n^s(I)$, $F_2\in \mathcal{T}_n^s(II)$, $F_3\in \mathcal{T}_n^s(III)$ and $F_4\in \mathcal{T}_n^s(IV)$. Since it is obtained by the construction of the IV-type, $F_4$ is necessarily disconnected. Contrary to Example~\ref{exa:H} this is a general one since it starts with any $F\in \mathcal{T}_{n-1}^{s-1}$.
\end{exa}

\vspace{3mm}
\begin{figure}[t]
\begin{center}
\unitlength 1mm % = 2.845pt
\linethickness{0.4pt}
\ifx\plotpoint\undefined\newsavebox{\plotpoint}\fi % GNUPLOT compatibility
\setlength{\unitlength}{3.8pt}
\begin{picture}(105,25.625)(0,0)
\put(11.693,15.527){\line(-1,0){.088}}
\put(11.279,15.997){\circle*{.94}}
\put(3.678,15.909){\circle*{.94}}
\put(7.804,15.837){\oval(15.026,4.861)[]}
\put(7.567,22.361){\circle*{.94}}
\put(11.125,16){\line(0,1){.25}}
%\emline(7.625,22.375)(4.375,18.25)
\multiput(7.625,22.375)(-.033505155,-.042525773){97}{\line(0,-1){.042525773}}
%\end
%\emline(7.625,22.375)(10.375,18.25)
\multiput(7.625,22.375)(.03353659,-.05030488){82}{\line(0,-1){.05030488}}
%\end
\put(6.625,15.625){\circle*{.25}}
\put(7.509,15.625){\circle*{.25}}
\put(8.393,15.625){\circle*{.25}}
\put(8.968,6.831){\circle*{.94}}
\put(5.914,6.831){\circle*{.94}}
\put(6.83,6.801){\circle*{.25}}
\put(7.356,6.801){\circle*{.25}}
\put(7.883,6.801){\circle*{.25}}
\put(7.563,6.625){\oval(8.125,3.5)[]}
\put(7.5,13.375){\line(0,-1){5.125}}
\put(7.688,13.438){\oval(15.375,20.375)[]}
\put(15.97,22.345){\circle*{.94}}
%\emline(15.875,22.25)(4.625,18.375)
\multiput(15.875,22.25)(-.097826087,-.033695652){115}{\line(-1,0){.097826087}}
%\end
%\emline(16,22.375)(10.375,18.5)
\multiput(16,22.375)(-.048913043,-.033695652){115}{\line(-1,0){.048913043}}
%\end
\put(67.093,6.706){\circle*{.94}}
\put(64.031,6.706){\circle*{.94}}
\put(64.956,6.676){\circle*{.25}}
\put(65.474,6.676){\circle*{.25}}
\put(66.004,6.676){\circle*{.25}}
\put(65.687,6.625){\oval(8.125,3.5)[]}
\put(61.191,18.658){\circle*{.94}}
\put(62.718,13.456){\circle*{.94}}
\put(59.652,13.456){\circle*{.94}}
\put(60.581,13.426){\circle*{.25}}
\put(61.096,13.426){\circle*{.25}}
\put(61.628,13.426){\circle*{.25}}
\put(61.313,13.375){\oval(8.125,3.5)[]}
\put(70.566,18.658){\circle*{.94}}
\put(72.093,13.456){\circle*{.94}}
\put(69.027,13.456){\circle*{.94}}
\put(69.956,13.426){\circle*{.25}}
\put(70.474,13.426){\circle*{.25}}
\put(71.004,13.426){\circle*{.25}}
\put(70.687,13.375){\oval(8.125,3.5)[]}
%\emline(61.119,18.75)(58.869,15)
\multiput(61.119,18.75)(-.03358209,-.05597015){67}{\line(0,-1){.05597015}}
%\end
%\emline(61.245,18.625)(63.245,15)
\multiput(61.245,18.625)(.03333333,-.06041667){60}{\line(0,-1){.06041667}}
%\end
%\emline(70.622,18.75)(67.869,14.875)
\multiput(70.622,18.75)(-.03357317,-.0472561){82}{\line(0,-1){.0472561}}
%\end
%\emline(70.495,18.75)(72.622,15.125)
\multiput(70.495,18.75)(.03325,-.05664063){64}{\line(0,-1){.05664063}}
%\end
\put(63.372,8){\line(0,1){0}}
\put(66.063,12.25){\oval(18.875,17.5)[]}
\put(66.093,23.095){\circle*{.94}}
%\emline(66.122,23.25)(68.122,15.125)
\multiput(66.122,23.25)(.03333333,-.13541667){60}{\line(0,-1){.13541667}}
%\end
%\emline(66.122,23)(72.498,15.125)
\multiput(66.122,23)(.033730159,-.041666667){189}{\line(0,-1){.041666667}}
%\end
\put(63.622,11.625){\line(0,-1){3.625}}
\put(68.37,11.75){\line(0,-1){3.75}}
\put(65.843,18.595){\circle*{.94}}
\put(65.998,18.625){\line(0,1){0}}
\put(61.122,18.625){\line(1,0){4.875}}
%\emline(65.87,18.625)(58.998,15.125)
\multiput(65.87,18.625)(-.066086538,-.033653846){104}{\line(-1,0){.066086538}}
%\end
%\emline(65.998,18.625)(63.122,15.125)
\multiput(65.998,18.625)(-.03343023,-.04069767){86}{\line(0,-1){.04069767}}
%\end
\put(65.748,18.625){\line(1,0){4.875}}
%\emline(66.122,23.25)(61.122,18.75)
\multiput(66.122,23.25)(-.037313433,-.03358209){134}{\line(-1,0){.037313433}}
%\end
%\emline(66.122,23)(70.87,18.5)
\multiput(66.122,23)(.035432836,-.03358209){134}{\line(1,0){.035432836}}
%\end
\put(37.468,6.581){\circle*{.94}}
\put(34.401,6.581){\circle*{.94}}
\put(35.331,6.551){\circle*{.25}}
\put(35.849,6.551){\circle*{.25}}
\put(36.379,6.551){\circle*{.25}}
\put(36.063,6.5){\oval(8.125,3.5)[]}
\put(31.566,18.533){\circle*{.94}}
\put(33.093,13.331){\circle*{.94}}
\put(30.026,13.331){\circle*{.94}}
\put(30.955,13.301){\circle*{.25}}
\put(31.47,13.301){\circle*{.25}}
\put(32.005,13.301){\circle*{.25}}
\put(31.688,13.25){\oval(8.125,3.5)[]}
\put(40.941,18.533){\circle*{.94}}
\put(42.468,13.331){\circle*{.94}}
\put(39.4,13.331){\circle*{.94}}
\put(40.331,13.301){\circle*{.25}}
\put(40.849,13.301){\circle*{.25}}
\put(41.378,13.301){\circle*{.25}}
\put(41.063,13.25){\oval(8.125,3.5)[]}
%\emline(31.491,18.625)(29.241,14.875)
\multiput(31.491,18.625)(-.03358209,-.05597015){67}{\line(0,-1){.05597015}}
%\end
%\emline(31.616,18.5)(33.615,14.875)
\multiput(31.616,18.5)(.03333333,-.06041667){60}{\line(0,-1){.06041667}}
%\end
%\emline(40.998,18.625)(38.242,14.75)
\multiput(40.998,18.625)(-.03360976,-.0472561){82}{\line(0,-1){.0472561}}
%\end
%\emline(40.869,18.625)(42.998,15)
\multiput(40.869,18.625)(.03325,-.05664063){64}{\line(0,-1){.05664063}}
%\end
\put(33.746,7.875){\line(0,1){0}}
\put(36.438,12.125){\oval(18.875,17.5)[]}
\put(36.468,22.97){\circle*{.94}}
%\emline(36.496,22.875)(29.246,14.875)
\multiput(36.496,22.875)(-.03372093,-.037209302){215}{\line(0,-1){.037209302}}
%\end
%\emline(36.496,23)(33.496,15)
\multiput(36.496,23)(-.03370787,-.08988764){89}{\line(0,-1){.08988764}}
%\end
%\emline(36.496,23.125)(38.496,15)
\multiput(36.496,23.125)(.03333333,-.13541667){60}{\line(0,-1){.13541667}}
%\end
%\emline(36.496,22.875)(42.872,15)
\multiput(36.496,22.875)(.033563158,-.041447368){190}{\line(0,-1){.041447368}}
%\end
\put(33.996,11.5){\line(0,-1){3.625}}
\put(38.746,11.625){\line(0,-1){3.75}}
\put(97.817,15.524){\line(-1,0){.088}}
\put(97.401,15.994){\circle*{.94}}
\put(89.801,15.906){\circle*{.94}}
\put(93.925,15.837){\oval(15.026,4.861)[]}
\put(93.691,22.361){\circle*{.94}}
\put(97.248,15.997){\line(0,1){.25}}
%\emline(93.748,22.373)(90.498,18.247)
\multiput(93.748,22.373)(-.033505155,-.042536082){97}{\line(0,-1){.042536082}}
%\end
%\emline(93.748,22.373)(96.498,18.247)
\multiput(93.748,22.373)(.03353659,-.05031707){82}{\line(0,-1){.05031707}}
%\end
\put(92.748,15.623){\circle*{.25}}
\put(93.631,15.623){\circle*{.25}}
\put(94.516,15.623){\circle*{.25}}
\put(95.093,6.831){\circle*{.94}}
\put(92.037,6.831){\circle*{.94}}
\put(92.953,6.798){\circle*{.25}}
\put(93.48,6.798){\circle*{.25}}
\put(94.007,6.798){\circle*{.25}}
\put(93.687,6.622){\oval(8.125,3.5)[]}
\put(93.622,13.372){\line(0,-1){5.125}}
\put(93.813,13.438){\oval(15.375,20.375)[]}
\put(102.093,22.343){\circle*{.94}}
\put(4.875,10.25){\makebox(0,0)[cc]{\scriptsize$F\in \mathcal{T}_{n-1}^{s-1}$}}
\put(32.5,9.875){\makebox(0,0)[cc]{\scriptsize$F\in \mathcal{T}_{n-1}^{s-1}$}}
\put(62.874,10){\makebox(0,0)[cc]{\scriptsize$F\in \mathcal{T}_{n-1}^{s-1}$}}
\put(92.999,10.625){\makebox(0,0)[cc]{\scriptsize$F\in \mathcal{T}_{n-1}^{s-1}$}}
\put(7.875,.125){\makebox(0,0)[cc]{\footnotesize$F_1\in \mathcal{T}_n^s(I)$}}
\put(36.124,0){\makebox(0,0)[cc]{\footnotesize$F_2\in \mathcal{T}_n^s(II)$}}
\put(65.999,.125){\makebox(0,0)[cc]{\footnotesize$F_3\in \mathcal{T}_n^s(III)$}}
\put(93.749,.125){\makebox(0,0)[cc]{\footnotesize$F_4\in \mathcal{T}_n^s(IV)$}}
\put(17.5,23.25){\makebox(0,0)[cc]{\scriptsize$u$}}
\put(9.25,21.625){\makebox(0,0)[cc]{\scriptsize$v$}}
\put(38.499,23.25){\makebox(0,0)[cc]{\scriptsize$u$}}
\put(30.125,17.875){\makebox(0,0)[cc]{\scriptsize$v$}}
\put(42.874,18.125){\makebox(0,0)[cc]{\scriptsize$w$}}
\put(67.749,23.625){\makebox(0,0)[cc]{\scriptsize$u$}}
\put(72.374,18.25){\makebox(0,0)[cc]{\scriptsize$v$}}
\put(66.249,16.625){\makebox(0,0)[cc]{\scriptsize$x$}}
\put(59.499,18.125){\makebox(0,0)[cc]{\scriptsize$y$}}
\put(103.749,20.875){\makebox(0,0)[cc]{\scriptsize$u$}}
\end{picture}
\end{center}

\caption{The graphs $F_1$, $F_2$, $F_3$ and $F_4$ (for Example~\ref{exa:e}).}\label{fig-7}
\end{figure}

\newpage

\section*{Acknowledgement}
The second and fifth authors are supported by the National Natural Science Foundation of China (Grant Nos.~11971274). The fourth author is supported by the Serbian Ministry of Education, Science and Technological Development via the University of Belgrade, Faculty of Mathematics.

\end{document}